\documentclass{article}

\usepackage{graphicx} 

\usepackage[a4paper, total={6in, 8in}]{geometry}
\usepackage{fancyhdr}
\usepackage{pdfpages}
\usepackage{hyperref}
\usepackage{caption}
\hypersetup{colorlinks = true, linkcolor = purple, citecolor= blue}
\usepackage[toc,page]{appendix}
\usepackage{dsfont}
\usepackage[utf8]{inputenc}
\usepackage{bbm}
\usepackage{amsmath}
\usepackage[normalem]{ulem}
\usepackage{amsfonts}
\usepackage{amssymb}
\usepackage{amsthm}
\usepackage{graphicx} 
\usepackage{subcaption}
\usepackage[toc,page]{appendix}
\newtheorem*{theorem*}{Theorem}
\newtheorem{theorem}{Theorem}[section]

\newtheorem{lemma}[theorem]{Lemma}

\newtheorem{remark}[theorem]{Remark}
\newtheorem*{remark*}{Remark}
\newtheorem{proposition}[theorem]{Proposition}
\usepackage{enumitem}

\usepackage[
  backend=biber,
  style=numeric,
  sortcites=true,
  defernumbers=true,
  giveninits=true,   
  eprint=false,
  url=false,        
  doi=false,         
  isbn=false        
]{biblatex}
\usepackage[section]{placeins}
\renewcommand{\epsilon}{\varepsilon}
\usepackage{float}

\usepackage{todonotes}

\title{Spectral asymptotics for a class of singular Sturm–Liouville operators with applications to magnetic Laplacian and $a$-zeros of Kummer functions}

\author{Roman Vanlaere \footnote{CEREMADE, Université Paris-Dauphine PSL, CNRS UMR 7534, 75016 Paris, France, roman.vanlaere@dauphine.psl.eu}}

\date{April, 2026}
\begin{document}

\maketitle

\begin{abstract}
    We provide a precise description of the bottom of the spectrum in the semiclassical limit of a harmonic-type Schrödinger operator with an inverse square potential. By exploiting the connection between the eigenfunctions of these operators and the Kummer and Whittaker functions, we derive accurate localization results for the non-asymptotic zeros of these functions with respect to their first parameter, uniformly with respect to the argument taken large and real. Moreover, our operators are linked to the magnetic Dirichlet Laplacian in the presence of both a
constant magnetic field and an Aharonov–Bohm flux line, so that our results describe its spectrum in the strong magnetic field limit. Our spectral analysis relies on a WKB-type approach.
\end{abstract}

\tableofcontents

\section{Introduction}

The goal of this paper is to give a precise description in the limit $\xi \to + \infty$ of the spectrum of the operators 
\begin{align*}
    -\partial_x^2 + \xi^2x^2 + \frac{\nu^2 - 1/4}{x^2}.
\end{align*}

These operators arise from the study of the magnetic Laplacian on a disk with Dirichlet boundary conditions, in the presence of both a constant magnetic field and an Aharonov–Bohm flux line through the origin, so that our results describe their spectrum in the strong field limit (see Section \ref{section: magnetic laplacian}). They are also closely related to Kummer and Whittaker functions, which appear in the expressions of their eigenfunctions, so that the spectral analysis of $G_\xi$ yields new results concerning the zeros, with respect to the first parameter, of this class of special functions (see Section \ref{section: zeros of kummer}). Finally, the operators $G_\xi$ correspond to the Fourier components, with respect to the $y$-variable, of the Grushin operator
\begin{align*}
    -\partial_x^2 - x^2 \partial_y^2 + \frac{\nu^2-1/4}{x^2},
\end{align*}
and the analysis of its spectrum has applications in control theory (see \cite{vanlaere2026observability}).

\subsection{Settings and main results}

Let $\xi > 0$ and $\nu \geq 0$. Set 
\begin{align*}
    G_\xi = -\partial_x^2 + \xi^2x^2 + \frac{\nu^2 - 1/4}{x^2},      
\end{align*}
defined on $C_c^\infty(0,1)$. Throughout the paper, $\nu \geq 0$ is a fixed parameter, while $\xi > 0$ is the semiclassical parameter with respect to which our results hold uniformly. We make an abuse of language here by calling $\xi$ the semiclassical parameter, but dividing $G_\xi$ by $\xi^2$, we are indeed in the semiclassical framework, with a singular semi-classical perturbation, as by setting $h = 1/\xi$,
\begin{align*}
    h^2G_\xi = -h^2\partial_x^2 + x^2 +h^2 \frac{\nu^2 - 1/4}{x^2}.
\end{align*}

We consider the Friedrich extension of $G_\xi$ in $L^2(0,1)$, and we keep the notation $G_\xi$. We first set $H_{0,\nu}^1(0,1)$ to be the completion of $C_c^\infty(0,1)$ with respect to the norm $\|\cdot\|_\nu$, where 
\begin{align}
    \|f\|_\nu := \left( \int_0^1 f'(x)^2 + \left( \xi^2x^2 + \frac{\nu^2 - 1/4}{x^2} \right)f(x)^2 \ dx\right)^{1/2}.
\end{align}
Recall the Hardy inequality (see \textit{e.g.} \cite{cannarsa2008carleman})
\begin{align}\label{eqn: hardy inequality}
    \int_0^1 \frac{u(x)^2}{x^2} \ dx \leq 4\int_0^1 u'(x)^2 \ dx, \quad \text{for every } u \in H^1((0,1),\mathbb{R}) \text{ with } u(0) = 0.
\end{align}

Thanks to \eqref{eqn: hardy inequality}, when $\nu > 0$, $\|\cdot\|_\nu$  and the usual $H_0^1(0,1)$-norm are equivalent. Thus, we have $H_{0,\nu}^1(0,1) = H_0^1(0,1)$. As shown in \cite[Section 5]{vazquez2000hardy}, in the critical case $\nu = 0$ we have the strict inclusion 
\begin{align*}
   H_0^1(0,1) \subsetneq H_{0,\nu}^1(0,1).
\end{align*}

We therefore consider the operators
\begin{align}\label{eqn: operator shrodinger singular}
    \begin{array}{ccl}
       D(G_\xi) &=& \left\{f \in H_{0,\nu}^1(0,1), \quad G_\xi f \in L^2(0,1) \right\},\\[6pt]
        G_\xi &=& \displaystyle -\partial_x^2 + \xi^2x^2 + \frac{\nu^2 - 1/4}{x^2}.
    \end{array}
\end{align}
We do not specify in the notations the dependence on $\nu$, since it is fixed and no confusion arises. These operators are positive-definite by Hardy inequality \eqref{eqn: hardy inequality}, self-adjoint with compact resolvent, and from Sturm-Liouville theory, their eigenvalues are simple. Hence, the eigenvalues form an increasing sequence of real numbers, 
\begin{align}
  0 \leq \lambda_{\xi,0} < \lambda_{\xi,1} < ... < \lambda_{\xi,k} < ...
\end{align}

The eigenfunctions of the operator $G_\xi$ defined by \eqref{eqn: operator shrodinger singular} will be denoted by $g_{\xi,k}$, and shown to be expressed in terms of a certain class of special functions that we now introduce. \\

The confluent hypergeometric functions are solutions of a class of linear second-order ODEs. Among them, the Kummer and Tricomi functions are solutions of the Kummer equation, for $a,b,z \in \mathbb{C}$, 
\begin{align}\label{eqn: Kummer equation}
    z w''(z) + \left( b - z \right) w'(z) - a w(z) = 0.
\end{align}
The Whittaker functions are solutions of the Whittaker equation, for $\kappa,\mu,z \in \mathbb{C}$,
\begin{align}\label{eqn: whittaker equation}
    W''(z) + \left( -\frac{1}{4} + \frac{\kappa}{z} - \frac{\mu^2 - 1/4}{z^2}\right)W(z) = 0.
\end{align}
These functions are studied for instance in \cite{buchholz2013confluent,OlverHandbook2010,bateman1953higher,tricomi1947sulle, slater1960confluent} and references therein, among others. \\

On one hand, the Kummer functions $M(a,b,z)$, solutions of \eqref{eqn: Kummer equation}, are defined, as long as $b$ is not a negative integer, by
\begin{align}\label{eqn: def Kummer function}
    M (a,b,z) &= \sum_{k \geq 0} \frac{(a)_k}{(b)_k k!}z^k, 
\end{align}
where $(a)_k = a(a+1)...(a+k-1)$, $(a)_0 = 1$, is the Pochhammer symbol. \\

On the other hand, the Whittaker functions, solutions of \eqref{eqn: whittaker equation}, are defined, as long as $2\mu$ is not a nonpositive integer, by
\begin{align}\label{eqn: definition whittaker function}
    M_{\kappa,\mu}(z) = e^{-\frac{z}{2}}z^{\frac{1}{2} + \mu}M\left(\frac{1}{2} + \mu - \kappa, 1 + 2\mu, z \right).
\end{align}

\subsubsection{A link between our objects of interest}

We have the following first result that links the eigenfunctions and eigenvalues of $G_\xi$ to the Kummer functions and their zeros with respect to the parameter $a$. 

\begin{theorem}\label{thm: exact form eigenfunctions singular + cond eigenv}
    Let $\nu \geq 0$. The eigenfunctions $g_{\xi,k}$ of $G_\xi $, $k \geq 0$, associated to $\lambda_{\xi,k}$, are  
    \begin{align}
    g_{\xi,k}(x) = A e^{-\xi x^2/2}x^{\frac{1}{2} + \nu}M \left(-\frac{\lambda_{\xi,k} - 2\xi (1 + \nu)}{4\xi }, 1 + \nu, \xi x^2 \right), \quad A \in \mathbb{R},
    \end{align}
where 
\begin{align}
    M \left(-\frac{\lambda_{\xi,k} - 2\xi (1 + \nu)}{4\xi },1 + \nu, \xi \right) = 0.
\end{align}
\end{theorem}

We prove Theorem \ref{thm: exact form eigenfunctions singular + cond eigenv} in Section \ref{section: link Kummer eigenfunction}.\\

In particular, what Theorem \ref{thm: exact form eigenfunctions singular + cond eigenv} states, is that $\lambda$ is an eigenvalue of $G_\xi$ if and only if $a$ defined by the identity
\begin{align}
    a = - \frac{\lambda - 2\xi(1+\nu)}{4\xi}, 
\end{align}
solves 
\begin{align}\label{eqn: main equation M = 0}
   M(a,b,\xi) = 0,
\end{align}
with $b = 1 + \nu$ and $\xi \in (0,+\infty)$. The following theorem is proved at the end of Section \ref{section: link Kummer eigenfunction}, and directly follows from Theorem \ref{thm: exact form eigenfunctions singular + cond eigenv}.

\begin{theorem}\label{thm: link zeros Kummer to eigenvalues}
    For every $b \geq 1$, that we write $b = 1 + \nu$ with $\nu \geq 0$, for every $\xi > 0$, $a_{\xi,k}$ is a solution of the equation $M(a,b,\xi) = 0$ if and only if 
    \begin{align}\label{eqn: identity a zeros eigenvalue}
        a_{\xi,k} = - \frac{\lambda_{\xi,k} - 2\xi(1+\nu)}{4\xi},
    \end{align}
    where $\lambda_{\xi,k}$ is an eigenvalue of the operator $(G_\xi,D(G_\xi))$. 
\end{theorem}

Therefore, equation \eqref{eqn: main equation M = 0} is of very much interest. Throughout this paper, the term $a$-zero will refer to a solution of \eqref{eqn: main equation M = 0} with respect to $a$, for fixed $b$ and $\xi$. Analogously, we shall also call a $z$-zero a solution of  $M(a,b,z) = 0$ with respect to $z$. There shall not be any confusion on the underlying function of interest, so we may use the term $\kappa$-zero for Whittaker functions. \\

From the definition of the Whittaker function \eqref{eqn: definition whittaker function}, it follows that $a_{\xi,k}$ is a solution of \eqref{eqn: main equation M = 0} if and only if 
\begin{align}\label{eqn: link between a zero kappa zero}
    \kappa_{\xi,k} = \frac{b}{2} - a_{\xi,k}
\end{align}
is a solution of
\begin{align}\label{eqn: equation kappa zero whittaker} 
     M_{\kappa,\mu}(\xi) = 0,
\end{align}
with $\mu = (b-1)/2$.\\

Thanks to Theorem \ref{thm: link zeros Kummer to eigenvalues}, the study of the $a$-zeros of Kummer functions, the $\kappa$-zeros of Whittaker functions, and the eigenvalues of $G_\xi$ are equivalent. One may therefore attempt to analyze Kummer and Whittaker functions directly in order to derive results for $G_\xi$ (as done in \cite{baur2025eigenvalues} for fixed $k$ and $\nu$ integer). However, the zeros with respect to the first parameter of this class of special functions are still not well understood, and the spectral asymptotic results for $G_\xi$ provided in Section \ref{section: spectral analysis} contribute to filling this gap. We refer the reader to Section \ref{section: zeros of kummer}, and in particular to Theorem \ref{thm: summarize a zeros Kummer}. \\

On the other hand, to the best of our knowledge, there is no clear description of the spectrum of $G_\xi$ in the limit $\xi \rightarrow + \infty$ in the literature when considering a bounded interval. Usually, such analysis are performed for a general potential that satisfies some sufficient regularity up to the boundary of the interval (see \textit{e.g.} \cite{voros1981spectre, simon1983semiclassical, helffer1984puits, Helffer1988-yw, allibert1998controle} among others). We also refer the reader to Section \ref{section: magnetic laplacian}, in which we explicit the link between our operators and the magnetic Dirichlet Laplacian, whose spectral analysis enjoys a vast literature. 

\subsubsection{Spectral analysis}\label{section: spectral analysis}

Regarding the spectrum of $G_\xi$ defined in \eqref{eqn: operator shrodinger singular}, we obtain the following results. 

\begin{theorem}\label{thm: lower bound eigenvalues}
     Let $\nu \geq 0$. For every $\xi > 0$ and every $k\geq 0$, we have
    \begin{align}\label{eqn: lower bound eigenvalues 1}
        \frac{\lambda_{\xi,k}}{\xi} > 4k + 2(1+\nu).
    \end{align}
    Moreover, there exists $c \in (0, \pi^2)$ and $\xi_0 > 0$, such that for every $\xi \geq \xi_0$, for every $k \geq 0$, 
    \begin{align}\label{eqn: lower bound eigenvalues 2}
        \lambda_{\xi,k} \geq ck^2.
    \end{align}
\end{theorem}

Theorem \ref{thm: lower bound eigenvalues} is proved in Section \ref{section: lower bound eigenvalues}. \\

For the first lower bound \eqref{eqn: lower bound eigenvalues 1} in Theorem \ref{thm: lower bound eigenvalues}, we provide two different proofs, respectively in Sections \ref{section: proof general first bound counting function} and \ref{section: proof general first bound minmax}. The first proof is based on the counting function of the $\xi$-zeros of the Kummer function. Namely, we actually prove a strict upper bound for the $a_{\xi,k}$ solutions of \eqref{eqn: main equation M = 0}, and use identity \eqref{eqn: identity a zeros eigenvalue} in Theorem \ref{thm: link zeros Kummer to eigenvalues}. Observe that the inequality is strict, although they are usually obtained in the large sense. The second proof uses a min-max argument. This second approach gives an inequality in the large sense, but using the idea of the first proof one recovers the strict inequality.\\

Theorem \ref{thm: upper bound eigenvalues lower range} below shows that the lower bound \eqref{eqn: lower bound eigenvalues 1} is optimal for a finite number, that depends on $\xi$, of eigenvalues. We loose the optimality when $k$ grows too large (we see this from \eqref{eqn: identity a zeros eigenvalue} combined with the asymptotic $a_{\xi,k} \sim -k^2\pi^2/4\xi$). Thus, we provide another lower bound \eqref{eqn: lower bound eigenvalues 2}, more optimal for large $k$, by using once again a min-max argument, but this time in a perturbative approach. The second lower bound \eqref{eqn: lower bound eigenvalues 2} is proved in Section \ref{section: proof general second bound}.\\  

We may already observe that the lower bound \eqref{eqn: lower bound eigenvalues 2} is better than \eqref{eqn: lower bound eigenvalues 1} as long as 
\begin{align*}
    k \geq \frac{4\xi}{c}\left(1 + \frac{b}{2k} \right).
\end{align*}

We now have the following exponentially precise upper bound for the eigenvalues of $G_\xi$ in the bottom of the spectrum.

\begin{theorem}\label{thm: upper bound eigenvalues lower range}
    Let $\tau \in (0,1)$. There exists $\xi_\tau > 0$, $C_1,C_2 > 0$, such that for every $\xi \geq \xi_\tau$, for every $k \leq \left\lfloor \frac{\tau \xi}{4} \right\rfloor$, we have
    \begin{align}
        \frac{\lambda_{\xi,k}}{\xi} \leq 4k + 2(1+\nu) + C_1 \,\xi e^{-C_2\xi }.
    \end{align}
\end{theorem}
 
The regime described in Theorem \ref{thm: upper bound eigenvalues lower range} is analogous to the usual low-energy regime in semiclassical analysis. What we actually compute is an upper bound for $|\lambda_{\xi,k}/\xi - 4k - 2(1+\nu)|$, which translates as stated in Theorem \ref{thm: upper bound eigenvalues lower range} due to the lower bound \eqref{eqn: lower bound eigenvalues 1} in Theorem \ref{thm: lower bound eigenvalues}. \\

Theorem \ref{thm: upper bound eigenvalues lower range} is proved in Section \ref{section: upper bound eigenvalues}. It is obtained through spectral theory of linear self-adjoint operators. We prove it by performing a WKB-type analysis of the eigenvalue problem associated with \eqref{eqn: operator shrodinger singular} in the semiclassical limit $\xi \rightarrow +\infty$. \\

More precisely, we construct quasi-modes $\phi_{\xi,k}$ (given in \eqref{eqn: def quasimodes}) for the operators under consideration, 
\begin{align*}
    G_\xi \phi_{\xi,k} \approx 4\xi k + 2\xi(1+\nu),
\end{align*}
with errors that are exponentially small in $L^2$-norm, and use a perturbative approach (see Lemma \ref{lemma: spectral distance}). The condition $\tau < 1$ guarantees that this error remains exponentially small. In the limiting case $\tau = 1$, our proof of Lemma \ref{lemma: upper bound eigenf of G low regime} given in Section \ref{section: proof technical lemmas} is no longer valid (we refer the reader in particular to \eqref{eqn: eigenf Plancherel-Rotach 2}). \\

In fact, upon a careful look at the computations (see in particular \eqref{eqn: application spectral distance to quasimode 2} in the proof of Theorem \ref{thm: upper bound eigenvalues lower range}, and \eqref{eqn: eigenf Plancherel-Rotach 2} in the proof of Lemma \ref{lemma: upper bound eigenf of G low regime}), we have that the constant $C_2$ gets closer to zero as $\tau$ is chosen closer to $1$. In the case of $k$ fixed, and in the limit $\xi \to +\infty$, precise asymptotics are obtained in \cite{baur2025eigenvalues,kachmar2025magnetic, fournais2024counting} for $\nu$ integer (see also Section \ref{section: magnetic laplacian}). In the case $k$ fixed, our strategy leads to the following explicit upper bound.

\begin{theorem}\label{thm: upper bound eigenvalues lower range fixed k}
    Let $k \in \mathbb{N}$. For every $\epsilon > 0$, there exists $\xi_0$ such that, for every $\xi \geq \xi_0$, we have 
    \begin{align}
        \frac{\lambda_{\xi,k}}{\xi} \leq 4k + 2(1+\nu) + (1+\epsilon)\sqrt{\frac{2}{\Gamma(1+\nu)}\frac{(1+\nu)_k}{k!}}\left( \frac{\xi}{\sqrt{6+2\nu}} +  \frac{4k+2(1+\nu)}{\sqrt{2(1+\nu)}} \right) \xi^{k + (1+\nu)/2} e^{-\xi/2},
    \end{align}
    where $\Gamma$ is the Gamma function, and $(a)_n$ is the Pochhammer symbol. 
\end{theorem}

We prove Theorem \ref{thm: upper bound eigenvalues lower range fixed k} in Section \ref{section: proof theorem upper bound fixed k}. However, the error term is less precise than in \cite{baur2025eigenvalues,kachmar2025magnetic} due to a weaker power in the exponential, but we are not restrained to $\nu$ integer. \\

The strategy described above may allow to handle, by adapting the computations in the proofs of Theorems \ref{thm: upper bound eigenvalues lower range} and \ref{thm: upper bound eigenvalues lower range fixed k}, a homogeneous Neumann boundary condition at $x = 1$, or Robin boundary condition $u'(1) = \gamma u(1)$, and lead to an analogous results as Theorem \ref{thm: upper bound eigenvalues lower range}, leaving the size of the regime for $k$ intact. The choice of the quasi-modes in \eqref{eqn: def quasimodes} has to be slightly modified, and we refer the reader to the associated footnotes for slight details. However, since $\lambda_{\xi,k}/\xi$ may become smaller than $4k + 2(1+\nu)$, the statement of the theorems will be changed to stating an exponentially small upper bound for $\operatorname{dist}(\lambda_{\xi,k}/\xi, 4k + 2(1+\nu))$. \\

We do not seek these improvement in the present work, and stick to the Dirichlet case. \\

Finally, the strategy of proof of Theorem \ref{thm: upper bound eigenvalues lower range} only allows for an upper bound, and not a lower bound like in \cite{baur2025eigenvalues, kachmar2025magnetic}, since we rely on Lemma \ref{lemma: spectral distance} from perturbation theory. To obtain a lower bound on the difference $(\lambda_{\xi,k}/\xi) - 4k - 2(1+\nu)$, one might have to combine our technical Lemmas (see Section \ref{section: proof theorem upper bound eigenvalues} and their proofs in Section \ref{section: proof technical lemmas}) with the variational approach of \cite{kachmar2025magnetic, fournais2024counting}.

\subsection{Zeros of Kummer and Whittaker functions with respect to the first parameter}\label{section: zeros of kummer}

Using the aforementioned link \eqref{eqn: identity a zeros eigenvalue} in Theorem \ref{thm: link zeros Kummer to eigenvalues} between the $a_{\xi,k}$, solutions of \eqref{eqn: main equation M = 0}, and the eigenvalues of the singular operator \eqref{eqn: operator shrodinger singular}, let us summarize in the following theorem what is obtained concerning the non-asymptotic $a$-zeros of equation \eqref{eqn: main equation M = 0}.
\begin{theorem}\label{thm: summarize a zeros Kummer}
    For every $b \geq 1$ and $\xi > 0$, the solutions of \eqref{eqn: main equation M = 0} form a strictly decreasing sequence of real numbers
    \begin{align}
        ... < a_{\xi,k} < a_{\xi,k-1} < ... < a_{\xi,1} < a_{\xi,0} < 0,
    \end{align}
    that satisfy
    \begin{align}\label{eqn: upper bound a zeros 1}
        a_{\xi,k} < -k, \quad \text{ for every } k \in \mathbb{N}.
    \end{align}
    Moreover, there exists $c \in (0, \pi^2)$ and $\xi_0 > 0$, such that for every $\xi \geq \xi_0$, for every $k \geq 0$, 
    \begin{align}\label{eqn: upper bound a zeros 2}
        a_{\xi,k} \leq -\frac{ck^2}{4\xi} + \frac{b}{2}.
    \end{align}
    Fix now $b \geq 1$. For every $\tau \in (0,1)$, there exists $\xi_\tau >0$, and $C_1,C_2 > 0$, such that for every $\xi \geq \xi_\tau$, for every $k \leq \lfloor\frac{\tau \xi}{4} \rfloor$,  
        \begin{align}
            - k - C_1e^{-C_2\xi} \leq a_{\xi,k} < -k.
        \end{align}
\end{theorem}

Using the identity \cite[Eq. (13.2.39)]{OlverHandbook2010}
\begin{align}\label{eqn: indentity positive negative roots}
    M(a,b,z) = e^zM(b-a,b,-z),
\end{align}
Theorem \ref{thm: summarize a zeros Kummer} extends to the case $\xi < 0$. \\

Moreover, using the definition \eqref{eqn: definition whittaker function} of Whittaker functions, from which is derived the identity \eqref{eqn: link between a zero kappa zero} between $a_{\xi,k}$ and $\kappa_{\xi,k}$, the above Theorem can be easily translated into a result for Whittaker functions. \\

Localizing the roots with respect to $z \in \mathbb{C}$ of the Kummer function \eqref{eqn: def Kummer function}, with $a,b \in \mathbb{C}$, or equivalently the roots of the Whittaker function \eqref{eqn: definition whittaker function} with $\kappa, \mu \in \mathbb{C}$, has attracted much more attention in the literature than finding the $a$-zeros, or $\kappa$-zeros (see e.g. \cite{ahmed1982properties, slater1960confluent, bateman1953higher, buchholz2013confluent,OlverHandbook2010}, or more recently \cite{boussaada2022some}). \\

To the best of our knowledge, there are no clear results in the literature concerning the solutions of \eqref{eqn: main equation M = 0} with respect to the parameter $a$, and the results presented in Section \ref{section: zeros of kummer} are new. It is known in the literature, for instance from \cite[Section 13.9]{OlverHandbook2010} and \cite[Section 17.2 page 185]{buchholz2013confluent}, but without a proof, or without a reference in the second case, that the solutions with respect to $a$, for $b,z \in \mathbb{C}$ fixed, of 
\begin{align}\label{eqn: main equation M=0 complex argument}
    M(a,b,z)  = 0,
\end{align}
asymptotically behave like $- k^2 \pi^2/4z$. Observe that this behavior is expected when the argument is real from the link \eqref{eqn: identity a zeros eigenvalue} between the $a_{\xi,k}$ and the eigenvalues of the singular operators \eqref{eqn: operator shrodinger singular}. But this is insufficient for our purpose as it does not cover the whole set of solutions, and these estimates lack uniformity with respect to the argument. Also, we will observe that the non-asymptotic solutions of \eqref{eqn: main equation M = 0} behave quite differently than this. \\

Again, when $b,z \in \mathbb{C}$ are fixed, by means of a Newton method, in \cite[Section 6.4, page 110]{slater1960confluent} the solutions with respect to $a$ of \eqref{eqn: main equation M=0 complex argument} are approximated. But these approximations are much less precise than the ones we will obtain. \\

More than being a natural question to try to understand the $a$-zeros of \eqref{eqn: main equation M = 0}, the latters have their importance as they are linked to the eigenvalues of a large class of second-order linear operators as shown by the identity \eqref{eqn: identity a zeros eigenvalue}. Indeed, Kummer and Whittaker functions often appear explicitly in the expression of the eigenfunctions of some differential operators (see e.g. \cite[Section 18]{buchholz2013confluent}, Theorem \ref{thm: exact form eigenfunctions singular + cond eigenv}, and Section \ref{section: magnetic laplacian}). Usually, to have a complete description of the spectrum of these operators, they are considered on the whole half-line, so that for $M(a,b,\xi)$ to be square integrable, one has to take $a = -k$ for $M$ to be polynomial (see Proposition \ref{prop: eigv and eigenf of G on R+}). From \eqref{eqn: identity a zeros eigenvalue}, this trivially forces the eigenvalues to be explicitly dependent on the positive integers. When one considers these eigenproblems on finite intervals, the eigenvalues are usually approximated numerically.

\subsection{The link with the magnetic Dirichlet Laplacian on a disk}\label{section: magnetic laplacian}

The operators $G_\xi$, introduced in \eqref{eqn: operator shrodinger singular}, arise from the magnetic Laplacian on a disk with Dirichlet boundary conditions, in the presence of both a
constant magnetic field and an Aharonov–Bohm flux line through the origin (see \textit{e.g.} \cite{exner2002generalized,helffer2025flux,helffer2025magnetic}). \\

Introduce formally the operator 
\begin{align}\label{eqn: dirichlet laplacian}
    \left(-i\nabla + A \right)^2, \quad A = \frac{B}{2} 
\begin{pmatrix}
 - y \\
x
\end{pmatrix}
+ \frac{\alpha}{x^2 + y^2}
\begin{pmatrix}
 - y \\
x
\end{pmatrix}
, 
\end{align}

on the disk $D(0,R) \subset \mathbb{R}^2$ of radius $R > 0$, where $B > 0$ is the constant magnetic field strength, and $\alpha$ is the
Aharonov–Bohm flux. We consider here the Friedrichs extension of \eqref{eqn: dirichlet laplacian} with domain $C_c^\infty(D(0,R))$. \\

The eigenvalue problem associated with this operator writes in polar coordinates, for $(r,\theta) \in (0,R) \times [0,2\pi)$,
\begin{align}
    \left[ \displaystyle -\partial_r^2 - \frac{1}{r}\partial_r - \frac{1}{r^2}\partial_\theta^2 - \frac{2i}{r}\left( \frac{Br}{2} + \frac{\alpha}{r}\right) \partial_\theta + \frac{B^2}{4}r^2 + B \alpha + \frac{\alpha^2}{r^2} \right] \varphi(r,\theta) = \lambda \varphi(r,\theta),
\end{align}

with boundary conditions given by $\varphi(R,\theta) = 0$, and we have $H_0^1 \left( D(0,R)\right)$-regularity. The eigenfunctions are of the form 
\begin{align*}
    \varphi(r,\theta) = \psi(r)e^{i\ell \theta}, \quad \ell \in \mathbb{Z}, \quad r \in (0,R), \quad \theta \in [0,2\pi),
\end{align*}

with $\psi$ satisfying
\begin{align}
    \left[ \displaystyle -\partial_r^2 - \frac{1}{r}\partial_r + \frac{(\ell + \alpha)^2}{r^2} + \frac{B^2}{4}r^2 + B ( \ell + \alpha ) \right] \psi(r) = \lambda \psi(r),
\end{align}
with $\psi(R) = 0$, and $H^1$-regularity at $r = 0$.\\

Now, setting $\psi(r) = \Psi(r)/\sqrt{r}$, the function $\Psi$ must satisfy
\begin{align*}
    \left[ - \partial_r^2 + \left( \frac{B}{2} \right)^2r^2 + \frac{(\ell+\alpha)^2-1/4}{r^2} \right] \Psi(r) = (\lambda - B(\ell + \alpha))\Psi(r), 
\end{align*}
with $\Psi(R) = \Psi(0) = 0$.\\

Observe that after the change of variable $x = r/R$, by setting $B = 2\xi/R^2$, we retrieve exactly the eigenvalue problem associated with our operators $G_\xi$ introduced in \eqref{eqn: operator shrodinger singular}, with 
\begin{align}
    \nu = |\ell + \alpha|,
\end{align}
so that our results described the spectrum of \eqref{eqn: dirichlet laplacian} in the strong magnetic field limit $B \rightarrow + \infty$. \\

Our parameter $\nu$ therefore encodes both the angular momentum $\ell$ and the
Aharonov–Bohm flux $\alpha$. When $\alpha \in (0,1)$, $\nu$ is non-integer, and reflects the presence of the singular flux line. This is the Aharonov–Bohm effect. \\

The spectrum of \eqref{eqn: dirichlet laplacian} has been intensively studied in the literature. However our results are, to the best of our knowledge, more precise than what is currently available, owing to both the precision and the uniformity of the estimates we provide. \\

We refer for instance the reader to \cite{montgomery1995hearing, helffer2017semi, frank2009polya, baur2025eigenvalues, helffer2025flux, helffer2025flux, kachmar2025magnetic, helffer2017semiclassical} and reference therein. \\

For instance, in \cite{helffer2025flux,helffer2025magnetic}, the study of the problem described above is performed on the exterior of a disk, respectively in the strong and weak field limit. In particular, in \cite{helffer2025flux}, a three-terms expansion for the lowest eigenvalue is given. We also refer the reader to \cite{kachmar2025laplace} for Neumann boundary conditions in this geometric setting in the weak field limit. On non simply-connected domains, we refer the reader to, for instance, \cite{helffer2017semiclassical}. \\

In \cite[Theorem 2.1]{baur2025eigenvalues}, they obtain similar results (translated in our context) as ours for $\lambda_{\xi,k}$ in the limit $\xi \rightarrow +\infty$, but with $k$ fixed and $\nu$ integer (\textit{i.e.} in the case of only a constant magnetic field), and with a different approach from ours. In \cite{kachmar2025magnetic}, similar results as in \cite{baur2025eigenvalues} are obtained for Neumann boundary conditions, via quasi-modes constructions and variational methods, which also apply to Dirichlet and Robin boundary conditions. Recall that, as stated at the end of Section \ref{section: spectral analysis}, our method to prove Theorem \ref{thm: upper bound eigenvalues lower range} can be adapted to these boundary conditions.

\subsection{Additional comments on the results}

We have identified, for any fixed $\nu \geq 0$, the low-lying eigenvalues, indexed by $k \leq \lfloor \tau \xi/4 \rfloor$ for some fixed $\tau \in (0,1)$, of $G_\xi$ in the limit $\xi \rightarrow + \infty$, and therefore the corresponding non-asymptotic solutions with respect to $a$ of \eqref{eqn: main equation M = 0} thanks to the identity \eqref{eqn: identity a zeros eigenvalue} in Theorem \ref{thm: exact form eigenfunctions singular + cond eigenv}. \\

One may wonder what to expect concerning the regime $k \geq \lfloor \tau \xi/4 \rfloor + 1$. In this regime, one should consider separately the case  $\lfloor \tau \xi/4 \rfloor + 1 \leq k \leq \lfloor \xi/4\tau \rfloor$, which is the intermediate regime, and $k \geq \lfloor \xi/4\tau \rfloor + 1$, which is the upper regime. These regimes correspond to the usual intervals of energy in the semiclassical framework. In the first case, the analysis should be performed in connection with the Airy equation, while in the second case in connection with the Bessel equation.\\

One may try to obtain results by means of a WKB analysis, for instance in the spirit of \cite{allibert1998controle} among others. But we do not expect a WKB analysis to be easy, due to the singular term $(\nu^2 - 1/4)/x^2$. Another possibility is to take advantage of Theorem \ref{thm: exact form eigenfunctions singular + cond eigenv}, in particular identity \eqref{eqn: identity a zeros eigenvalue}, and look for uniform asymptotic expansions of Kummer and Whittaker functions as explained below.\\

In the case of use of asymptotic expansions, one should look for uniform, with respect to the argument and the first parameter, asymptotic expansion of Kummer or Whittaker functions. Let us formally explain how this approach would work. We drop the $x$ variable by fixing $x=1$, and look at the semiclassical parameter $\xi$ as the new argument. That is, we try to solve \eqref{eqn: main equation M = 0} or \eqref{eqn: equation kappa zero whittaker} with respect to the first parameter by looking at $\xi$ as a real argument in its own right. Then, we use the expansions obtained in \cite[Section 10]{erdelyi1957asymptotic} for instance. The expansions obtained in \cite{erdelyi1957asymptotic} are for the Whittaker functions, but by their definition \eqref{eqn: definition whittaker function}, and Theorem \ref{thm: exact form eigenfunctions singular + cond eigenv}, they are linked to $G_\xi$. \\

The advantage of the expansions in \cite{erdelyi1957asymptotic} is that the first parameter $\kappa$ of the Whittaker function \eqref{eqn: definition whittaker function} is moved into the argument of the main term of the expansion, and thanks to a Liouville-Green transformation, in some sense the problem is reduced to investigate the zeros with respect to the argument of Airy or Bessel functions, from which localization results for the $\kappa$-zeros may be derived.\\

The expansion \cite[(10.1) page 31]{erdelyi1957asymptotic} in terms of Bessel functions should give a result concerning the solutions of \eqref{eqn: equation kappa zero whittaker} with respect to $\kappa$, uniformly with respect to $\xi$, and therefore a result concerning the eigenvalues of $G_\xi$. To do so, one should prove that these solutions are located near the zeros of the main term in the expansion using Rouché's theorem. This would lead to an implicit description of the spectrum and $a$-zeros in this regime, but with also a gap property. \\

However, we do not expect the expansion \cite[(10.3) page 31]{erdelyi1957asymptotic} in terms of Airy functions to work. Indeed, the expansion does not extend to the complex plane, so Rouché's theorem cannot be used. Moreover, an analysis of the polar decomposition of Airy functions shows that the reminder term is too big in the limit $\xi \rightarrow + \infty$, making it hard to show that the zeros of the main term of the expansion approximate correctly the $\kappa$-zeros. Moreover, one needs to make sure that no zeros have been forgotten in this regime, which is an additional difficulty. Observe that this is not a problem in the upper regime thanks to Rouché's theorem. \\

We stress that asymptotic expansions that are also uniform with respect to the second parameter are given in \cite{dunster89}, and are completely analogous to those obtained in \cite{erdelyi1957asymptotic}. However, these are much harder to manipulate, as the change of variables by Liouville-Green transformations are implicit. \\

Finally, one may replace in $G_\xi$ defined by \eqref{eqn: operator shrodinger singular} the harmonic potential $\xi^2x^2$ by $\xi^2q(x)^2$, for a regular function $q$ that satisfies $q(x) \sim q'(0)x$ near zero. Then, using perturbation theory for self-adjoint operators (\textit{e.g.} Lemma \ref{lemma: spectral distance} or \cite[Proposition 4.1]{darde2023null}), one may obtain a result concerning the low-lying eigenvalues of the newly obtained generalized operator. We also refer the reader to \cite[Section 3]{helffer2025quantum} where this study is performed on the plane, with $\nu \in [0,1/2)$. \\

However, we stick in the present paper to the low-lying eigenvalues of our toy model $G_\xi$, and do not seek the improvement presented above. 

\subsection{Structure of the paper}

Let us summarize the general structure of the paper. In Section \ref{section: link Kummer eigenfunction}, we prove Theorem \ref{thm: exact form eigenfunctions singular + cond eigenv} and Theorem \ref{thm: link zeros Kummer to eigenvalues}. Namely, the one-to-one correspondence between the eigenvalues of $G_\xi$ and the $a_{\xi,k}$ solutions of \eqref{eqn: main equation M = 0}. Then, in Section \ref{section: lower bound eigenvalues}, we prove Theorem \ref{thm: upper bound eigenvalues lower range}. The two proofs of the first bound \eqref{eqn: lower bound eigenvalues 1} in Theorem \ref{thm: upper bound eigenvalues lower range} are given in Sections \ref{section: proof general first bound counting function} and \ref{section: proof general first bound minmax}, while the second bound \eqref{eqn: lower bound eigenvalues 2} is proved in Section \ref{section: proof general second bound}. Finally, Section \ref{section: upper bound eigenvalues} is devoted to the proof of Theorem \ref{thm: upper bound eigenvalues lower range}. We prove Theorem \ref{thm: upper bound eigenvalues lower range} in Section \ref{section: proof theorem upper bound eigenvalues} using some technical lemmas that are proved in Section \ref{section: proof technical lemmas}.

\section{Kummer functions as eigenfunctions of a class of Schrodinger operators: proof of Theorem \ref{thm: exact form eigenfunctions singular + cond eigenv}}\label{section: link Kummer eigenfunction}

The main purpose of this section is to link the Kummer functions to the eigenfunctions of $G_\xi$ introduced in \eqref{eqn: operator shrodinger singular}. We prove Theorem \ref{thm: exact form eigenfunctions singular + cond eigenv}.\\ 

We first show that the solutions of $G_\xi f = Ef$ write in terms of solutions of the Kummer equation \eqref{eqn: Kummer equation}. Before giving the Proposition, let us remind the couple of linearly independent solutions of the Kummer equation \eqref{eqn: Kummer equation} that we might encounter. The first solution will always be given, in our cases, by the Kummer function introduced in \eqref{eqn: def Kummer function}. We denote the second fundamental solution of \eqref{eqn: Kummer equation} of interest to us, for $b \in \mathbb{R}$ and $a,z \in \mathbb{C}$, by $\mathbf{M}(a,b,z)$. Its definition depends on wether or not the parameter belongs to some subsets of $\mathbb{R}$. Following for instance \cite[table p. 9]{mathews2021physicist}, or \cite[Section 13.2]{OlverHandbook2010}, we have the following. \\

\begin{enumerate}[label = (\roman*)]
    \item If $b \notin \mathbb{Z}$, $\mathbf{M}(a,b,z)$ is defined by
\begin{align}\label{eqn: second fundamental solution Kummer 1}    
     z^{1-b}M (a+1-b, 2-b, z).
\end{align}
    \item If $b = 1 + n$, $n \in \mathbb{N}$, and $n-a \notin \mathbb{N}$, $\mathbf{M}(a,b,z)$ is defined by
    \begin{align}\label{eqn: second fundamental solution Kummer 2}
    \sum_{k=1}^{n} \frac{n! (k-1)!}{(\nu - k)! (1 - a)_k} z^{-k} - \sum_{k=0}^{\infty} \frac{(a)_k}{(n + 1)_k \, k!} z^k \left( \ln z + \psi(a + k) - \psi(1 + k) - \psi(n + k + 1) \right),
    \end{align}
    where $\psi$ is the digamma function $\Gamma'(z)/\Gamma(z)$, and $\Gamma$ is the Gamma  function defined by 
    \begin{align*}
    \Gamma(z) = \int_0^\infty e^{-t}t^{z-1} \ dt.
    \end{align*}
    \item If $b = 1 + n$, $n\in \mathbb{N}$, and $-a \in \mathbb{N}$, $\mathbf{M}(a,b,z)$ is defined by 
    \begin{align}\label{eqn: second fundamental solution Kummer 3}
    \sum_{k=1}^{n} \frac{n! (k-1)!}{(\nu - k)! (1 - a)_k} z^{-k} &- \sum_{k=0}^{-a} \frac{(a)_k}{(n + 1)_k \, k!} z^k \left( \ln z + \psi(a + k) - \psi(1 + k) - \psi(n + k + 1) \right) \notag \\
    &+ (-1)^{1-a}(-a)! \sum_{k = 1 - a}^\infty \frac{(k-1+a)!}{(n+1)_k k!}z^k,
    \end{align}
    where $\psi$ is the Digamma function. \\

    \item If $b = 1 + n$, $n\in \mathbb{N}$, and $a \in \mathbb{N}$, $\mathbf{M}(a,b,z)$ is defined by
    \begin{align}\label{eqn: second fundamental solution Kummer 4}
    \sum_{k=a}^n \frac{(k-1)!}{(n-k)!(k-a)!}z^{-k}.
    \end{align}
\end{enumerate}

\begin{proposition}\label{prop: form eignfunction}
     Let $E \geq 0$. The solutions of 
\begin{align}\label{eqn: general eigenequation}
    -f''(x) + \xi^2x^2f(x) + \frac{\nu^2 - 1/4}{x^2} f - E f = 0
\end{align}
are of the form 
\begin{align}\label{eqn : eigenf candidates}
      f_E(x) = e^{-\xi x^2/2}x^{\frac{1}{2} + \nu} \left( A M \left(a,b,\xi x^2 \right) + B \mathbf{M}\left(a,b, \xi x^2 \right) \right), 
\end{align}
where $M$ is the Kummer function defined in \eqref{eqn: def Kummer function}, $\mathbf{M}$ is defined by either \eqref{eqn: second fundamental solution Kummer 1}, \eqref{eqn: second fundamental solution Kummer 2}, \eqref{eqn: second fundamental solution Kummer 3} or \eqref{eqn: second fundamental solution Kummer 4}, the coefficients $A$ and $B$ are real, and the parameters are defined by
\begin{align}\label{eqn: parameter candidates}
    \begin{array}{ccl}
       a &=&  \displaystyle -\frac{E - 2\xi(1 + \nu)}{4\xi}, \\[8pt]
       b &=& 1 + \nu. 
    \end{array}
\end{align}
\end{proposition}

\begin{proof}
We make the ansatz $f(x) = e^{-\xi x^2/2}g(x)$. The goal of this ansatz is to deal with the harmonic potential $\xi^2 x^2$. We have 
\begin{align*}
    f'(x) &= g'(x)e^{-\xi x^2/2} - \xi xg(x)e^{-\xi x^2/2}, \\
    f''(x) &= g''(x)e^{-\xi x^2/2} - 2\xi xg'(x)e^{-\xi x^2/2} - ng(x)e^{-\xi x^2/2} + \xi^2 x^2g(x)e^{-\xi x^2/2}.
\end{align*}
It follows that $g$ must satisfy 
\begin{align}
    - g''(x) + 2\xi xg'(x) + \xi g(x) + \frac{\nu^2 - 1/4}{x^2}g(x) - Eg(x) = 0.
\end{align}
We now make the transformation $g(x) = x^\alpha h(x)$, for some $\alpha > 0$ to be determined, which is coherent with the Frobenius method. We get
\begin{align*}
    g'(x) &= \alpha x^{\alpha - 1}h(x) + h'(x)x^\alpha, \\
    g''(x) &= \alpha(\alpha-1)x^{\alpha - 2}h(x) + 2\alpha h'(x)x^{\alpha - 1} + h''(x)x^\alpha.
\end{align*}
Plugging into the equation we get that 
\begin{align*}
    x^\alpha \left[ - \frac{\alpha(\alpha - 1)}{x^2}h(x) - \frac{2\alpha}{x}h'(x) - h''(x) + 2\xi \alpha h(x) + 2\xi x h'(x) +  n h(x) + \frac{\nu^2 - 1/4}{x^2}h(x) - E h(x)\right] = 0.
\end{align*}

Choosing $\alpha = \frac{1}{2} + \nu$, which is the largest positive root of $ -\alpha(\alpha - 1) + \nu^2 - 1/4 = 0$, we obtain that $h$ is solution of 
\begin{align*}
    h''(x) + \left( \frac{1 + 2\nu}{x} - 2\xi x \right) h'(x) + \left( E - 2\xi (1 + \nu) \right) h(x) = 0.
\end{align*}

Now, setting $z = \xi x^2$, $\tilde{h}(z) = h(x)$, we get that
\begin{align*}
    h'(x) &= 2\xi x\tilde{h}'(z), \\
    h''(x) &=  2\xi \tilde{h}'(z) + 4\xi^2 x^2\tilde{h}''(z),
\end{align*}
so that $\tilde{h}$ must solve 
\begin{align*}
     2\xi \tilde{h}'(z) + 4\xi^2 x^2\tilde{h}''(z) + \left( \frac{1 + 2\nu}{x} - 2\xi x \right)2\xi x\tilde{h}'(z) + (E - 2\xi (1 + \nu)) \tilde{h}(z) = 0,
\end{align*}
or equivalently 
\begin{align*}
    &4\xi z \tilde{h}''(z) + 2\xi \tilde{h}'(z) +  2\xi (1 + 2\nu)\tilde{h}'(z) - 4\xi z \tilde{h}'(z) + \left( E - 2\xi (1 + \nu) \right)\tilde{h}(z) = 0,
\end{align*}
or equivalently
\begin{align*}
    z \tilde{h}''(z) + \left( 1 + \nu - \xi \right) \tilde{h}'(z) +  \frac{E - 2\xi (1 + \nu)}{4\xi }\tilde{h}(z) = 0.
\end{align*}

This is the Kummer equation introduced in \eqref{eqn: Kummer equation}, and whose first fundamental solutions is given by the Kummer function \eqref{eqn: def Kummer function}, and the second fundamental solution is defined by either \eqref{eqn: second fundamental solution Kummer 1}, \eqref{eqn: second fundamental solution Kummer 2}, \eqref{eqn: second fundamental solution Kummer 3} or \eqref{eqn: second fundamental solution Kummer 4}. Substituting back $z = \xi x^2$, it concludes the proof. 
\end{proof}

We can now give conditions on the coefficients $A,B$ for $f_E$ to be in $D(G_\xi )$, and see the implied necessary and sufficient conditions on the eigenvalues.

\begin{proof}[Proof of theorem \ref{thm: exact form eigenfunctions singular + cond eigenv}]
Let us denote in this proof the first and second fundamental solutions of equation \eqref{eqn: general eigenequation} in Proposition \ref{prop: form eignfunction} respectively by 
\begin{align}
    \varphi_1(x) &= e^{-\xi \frac{x^2}{2}}x^{\frac{1}{2} + \nu}M(a,b,\xi x^2), \\
    \varphi_2(x) &= e^{-\xi \frac{x^2}{2}}x^{\frac{1}{2} + \nu}\mathbf{M}(a,b,\xi x^2),
\end{align}
so that a solution of \eqref{eqn: general eigenequation} writes $f_E(x) = A\varphi_1(x) + B\varphi_2(x)$. For simplicity of the presentation, we omit the computational details. Recall that $a$ and $b$ are defined by \eqref{eqn: parameter candidates} in Proposition \ref{prop: form eignfunction}.\\

Let us first investigate the case $\nu > 0$. Recall that in this case, $H_{0,\nu}^1(0,1) = H_0^1(0,1)$. If $\nu \notin \mathbb{N}$, then $\varphi_2$ is defined using \eqref{eqn: second fundamental solution Kummer 1}. It is not hard to observe that $\varphi_2 \notin H^1(0,1)$, so that we must have $B = 0$. On the other hand, $\varphi_1 \in H^1(0,1)$ and $\varphi_1(0)=0$, so we must only check the boundary condition at $x = 1$. We have $\varphi_1(1) = 0$ if and only if $M(a,b,\xi) = 0$, which concludes the proof in this case.  \\

Now, assume that $\nu \in \mathbb{N} \setminus \{0\}$. In this case $\varphi_2$ may be defined by either \eqref{eqn: second fundamental solution Kummer 2}, \eqref{eqn: second fundamental solution Kummer 3} or \eqref{eqn: second fundamental solution Kummer 4}. We may observe that the argument presented above works in this case too, due to the singular behavior of $\varphi_2$ as $x \rightarrow 0^+$.\\

Now, if $\nu = 0$, we recall that we have $H_0^1(0,1) \subsetneq H_{0,\nu}^1(0,1)$. In this case, $\varphi_2$ is again defined using either \eqref{eqn: second fundamental solution Kummer 2}, \eqref{eqn: second fundamental solution Kummer 3} or \eqref{eqn: second fundamental solution Kummer 4}. Observe that neither $\varphi_1$ or $\varphi_2$ belong to $H^1(0,1)$. But one may verify that 
\begin{align*}
    \int_0^1 \varphi_2'(x)^2 - \frac{1}{4x^2} \varphi_2(x)^2 \ dx = + \infty,\\[6pt]
    \int_0^1 \varphi_1'(x)^2 - \frac{1}{4x^2} \varphi_1(x)^2 \ dx < + \infty,
\end{align*}
using for instance the behavior of these functions as $x \rightarrow 0^+$. Hence, we must have $B = 0$. Once again, for $\varphi_1$ to belong to $H_{0,\nu}^1(0,1)$, the only thing to check is the boundary condition at $x = 1$ since $\varphi_1(0) = 0$. That is verified if and only if $M(a,b,\xi)= 0$, which concludes the proof.
\end{proof}

We can now prove Theorem \ref{thm: link zeros Kummer to eigenvalues}. 

\begin{proof}[Proof of Theorem \ref{thm: link zeros Kummer to eigenvalues}]
     It directly follows from Theorem \ref{thm: exact form eigenfunctions singular + cond eigenv}. Let $b \geq 1$, that we may write $b = 1 + \nu$ for some $\nu \geq 0$, and $\xi > 0$. We can also make the choice to write $a \in \mathbb{C}$ as, for some $\xi > 0$,
    \begin{align*}
        a = - \frac{E - 2\xi(1+\nu)}{4\xi}, 
    \end{align*}
    where $E$ ranges in $\mathbb{C}$. Hence, $a$ is a solution of $M(a,b,\xi) = 0$ if and only if 
    \begin{align*}
        M\left( - \frac{E - 2\xi(1+\nu)}{4\xi}, 1+\nu, \xi \right) = 0.
    \end{align*}
    Thanks to Proposition \ref{prop: form eignfunction} and \ref{thm: exact form eigenfunctions singular + cond eigenv}, this is equivalent to saying that $E$ is an eigenvalue of $G_\xi$.
\end{proof}

\section{A general lower bound on the eigenvalues: proof of Theorem \ref{thm: lower bound eigenvalues}}\label{section: lower bound eigenvalues}

Thanks to Theorem \ref{thm: link zeros Kummer to eigenvalues}, to give an upper bound on the $a$-zeros of $M(a,b,\xi)$ is equivalent to giving a lower bound on the eigenvalues of $G_\xi$. We have two ways of proving the first lower bound \eqref{eqn: lower bound eigenvalues 1} in Theorem \ref{thm: lower bound eigenvalues}.\\

The first one is to actually prove the first upper bound \eqref{eqn: upper bound a zeros 1} in Theorem \ref{thm: summarize a zeros Kummer} and use the identity \eqref{eqn: identity a zeros eigenvalue} in Theorem \ref{thm: link zeros Kummer to eigenvalues}. To prove \eqref{eqn: upper bound a zeros 1}, we use the counting function on the $z$-zeros of the Kummer functions. The second possibility is to use a min-max argument.\\

We stress that the second approach gives an inequality in the large sense, and to obtain a strict inequality, one must slightly rely on the first proof. However, we still choose to disclose both approach here, since the second one will be of use in Section \ref{section: upper bound eigenvalues}.\\ 

The first and second proofs of the first lower bound \eqref{eqn: lower bound eigenvalues 1} in Theorem \ref{thm: link zeros Kummer to eigenvalues} are given respectively in Section \ref{section: proof general first bound counting function} and \ref{section: proof general first bound minmax}.\\

The second upper bound \eqref{eqn: lower bound eigenvalues 2} in Theorem \ref{thm: link zeros Kummer to eigenvalues} follows from a perturbative argument in the sense of quadratic forms for the operator $G_\xi$, using the Rayleigh formula. We prove it in Section \ref{section: proof general second bound}.

\subsection{First proof of the first bound \eqref{eqn: lower bound eigenvalues 1} using the counting function}\label{section: proof general first bound counting function}

We have the following Lemma on the number of $z$-zeros of $M(a,b,z)$. 

\begin{lemma}{\cite[13.9(i)]{OlverHandbook2010}}\label{lemma: counting function}
 Let $a \in \mathbb{R}$, $b \geq 0$, and $z \in \mathbb{R}$. Then, the number of positive $z$-zeros of $M(a,b,z)$, denoted by $p(a,b)$, is finite and verifies
 \begin{align}
    \begin{array}{lccl}
     p(a,b) &=& 0,& \quad \text{ if } a \geq 0, \\
     p(a,b) &=& \lceil -a \rceil,& \quad \text{ if } a < 0,
     \end{array}
 \end{align}
 where $\lceil -a \rceil$ is the smallest integer greater than $-a$.
\end{lemma}

We can give the first proof of \eqref{eqn: lower bound eigenvalues 1}.

\begin{proof}[First proof of the first lower bound \eqref{eqn: lower bound eigenvalues 1} in Theorem \ref{thm: lower bound eigenvalues}]
    Recall from Theorem \ref{thm: link zeros Kummer to eigenvalues}, that $a_{\xi,k}$ is a $a$-zero of $M(a,b,\xi)$, with $b = 1 + \nu$, $\nu \geq 0$, if and only if 
    \begin{align}
        a_{\xi,k} = - \frac{\lambda_{\xi,k} - 2\xi(1+\nu)}{4\xi},
    \end{align}
    where $\lambda_{\xi,k}$ is an eigenvalue of the operator $(G_\xi,D(G_\xi))$. By Sturm-Liouville oscillation theorem, the number of zeros in $[0,1]$ of the eigenfunction $g_{\xi,k}$, for $k \geq 0$, is $k + 2$. In particular, due to the form of the eigenfunctions from Proposition \ref{thm: exact form eigenfunctions singular + cond eigenv}, outside $x=0$, we have $g_{\xi,k}(x) = 0$ if and only if $M(a_{\xi,k},1+\nu,\xi x^2) = 0$. It follows that $a_{\xi,k}$ is a $a$-zero if and only if $M(a_{\xi,k},1+\nu,\xi x^2)$ has at least $k+1$ $x$-zeros on $\mathbb{R}^+$, since $M(a,b,0) = 1$. Hence, from Lemma \ref{lemma: counting function}, we must have $p(a_{\xi,k},b) \geq k+1$, which is equivalent to saying that $-a_{\xi,k} > k$. Expanding $a_{\xi,k}$ in term of $\lambda_{\xi,k}$ by identity \eqref{eqn: identity a zeros eigenvalue}, we get the bound on the eigenvalues.     
\end{proof}

\subsection{Second proof of the first bound \eqref{eqn: lower bound eigenvalues 1} using the min-max theorem}\label{section: proof general first bound minmax}

Let us now give a proof of the first lower bound \eqref{eqn: lower bound eigenvalues 1} in Theorem \ref{thm: lower bound eigenvalues} using Rayleigh formula. First, set $\Omega_\xi  = (0,\sqrt{\xi})$. For $x \in (0,1)$, we make the change of variable $y = \sqrt{\xi}x$ and $u(x) = v(\sqrt{\xi}x)$. Thus, $v$ satisfies
\begin{align*}
    G_\xi  v(y) = \xi \left(-\partial_y^2 v(y) + y^2 v(y) + \frac{\nu^2 - 1/4}{y^2}v(y) \right). 
\end{align*}
Thus, we set
\begin{align}\label{eqn: definition tilde G_xi}
    \begin{array}{crl}
    D(\tilde{G}_\xi ) &:=& \{ v \in H_{0,\nu}^1(\Omega_\xi), \quad \tilde{G}_\xi v \in L^2(\Omega_\xi) \}, \\[6pt]
    \tilde{G}_\xi  &=& -\partial_y^2 + y^2 + \frac{\nu^2 - 1/4}{y^2},      
    \end{array}
\end{align}
with eigenvalues $\tilde{\lambda}_{\xi,k}$ which therefore satisfy
\begin{align*}
    \tilde{\lambda}_{\xi,k} = \frac{\lambda_{\xi,k}}{\xi}.
\end{align*}

Now, denote by $(G,D(G))$ the differential operator 
\begin{align}
\begin{array}{ccl}
    D(G) &=& \{ v \in H_{0,\nu}^1(\mathbb{R^+}), \quad Gv \in L^2(\mathbb{R^+}) \}, \\[6pt]
    G &=& -\partial_x^2 + x^2 + \frac{\nu^2 - 1/4}{x^2}, 
\end{array}
\end{align}
and denote its eigenvalues by $\mu_k$, $k\geq 0$. Let us describe precisely the eigenvalues of $G$ and the associated eigenfunctions, which are already known (see \textit{e.g.} \cite{exner2002generalized}), but we recall the proof.

\begin{proposition}\label{prop: eigv and eigenf of G on R+}
    The eigenfunctions of $G$ are of the form 
    \begin{align}
        \Phi_k(x) = Ae^{-x^2/2}x^{1/2 + \nu}M(-k,1+\nu,x^2), \quad A \in \mathbb{R},
    \end{align}
    associated to the eigenvalues
    \begin{align}\label{eqn: eigenvalues of G}
        \mu_k = 4k + 2(1+\nu), \quad k \geq 0.
    \end{align}
\end{proposition}

\begin{proof}
    The exact same computations as for Proposition \ref{prop: form eignfunction} in the case $\xi = 1$, with the same arguments as for the proof of Theorem \ref{thm: exact form eigenfunctions singular + cond eigenv}, give that an eigenfunction of $G$ must be of the form
    \begin{align*}
        \Phi_E(x) = e^{-x^2/2} x^{\frac{1}{2} + \nu}M \left(a,b,x^2 \right),
    \end{align*}
    with $a$, $b$ defined in Proposition \ref{prop: form eignfunction} (with $\xi=1$) and $a$ well chosen as follows. \\
    
    Since $\Phi_E(0) = 0$ for every $E \in \mathbb{R}$, $\Phi_E$ is an eigenfunction if and only if it belongs to $L^2(\mathbb{R}^+)$. The asymptotic for large $x$ of the confluent function $M(a,b,\cdot)$, under the restriction that $a \neq 0, -1, ...$, is given by (\cite[(13.7.1)]{OlverHandbook2010})
    \begin{align*}
        M(a,b,x^2) = e^{x^2}\frac{x^{a-b}}{\Gamma(a)\Gamma(b)}\left( 1 + O\left(\frac{1}{x} \right) \right).
    \end{align*}
    On the other hand, if $a = -k$, $k \in \mathbb{N}$, then $M(a,b,\cdot)$ is a polynomial of degree $k$. It follows that $\Phi_E$ is an eigenfunction if and only if $a = -k$, that is $E = 4k + 2(1+\nu)$.
\end{proof}

Therefore, by Rayleigh formula, we have the following second proof of the first lower bound \eqref{eqn: lower bound eigenvalues 1} in Theorem \ref{thm: lower bound eigenvalues}

\begin{proof}[Second proof of the first lower bound \eqref{eqn: lower bound eigenvalues 1} in Theorem \ref{thm: lower bound eigenvalues}]
    This is simply an application of Rayleigh formula. Denote by $\tilde{E}_{n,k} \subset L^2(0,+\infty)$ the vector space of dimension $k$ spanned by the normalized eigenfunctions $\tilde{f}_{n,1},...,\tilde{f}_{n,k}$ of $\tilde{G}_\xi $, extended by zero outside $\Omega_\xi$. We have 
    \begin{align*}
        \frac{\lambda_{\xi,k}}{\xi} = \tilde{\lambda}_{\xi,k} &= \underset{\underset{\|f\|_{L^2(\Omega_\xi )}=1}{f \in \tilde{E}_{n,k}}}{\max} \int_{\Omega_\xi } f'(x)^2 + x^2f(x)^2 + \frac{\nu^2-1/4}{x^2}f(x)^2 \ dx \\
        &= \underset{\underset{\|f\|_{L^2(\mathbb{R}^+)}=1}{f \in \tilde{E}_{n,k}}}{\max} \int_0^{+\infty} f'(x)^2 + x^2f(x)^2 + \frac{\nu^2-1/4}{x^2}f(x)^2 \ dx \\
        &\geq \underset{\underset{\operatorname{dim}(V) = k}{V \subset L^2(\mathbb{R}^+)}}{\min} \underset{\underset{\|f\|_{L^2(\mathbb{R}^+)}=1}{f \in V}}{\max} \int_0^{+\infty} f'(x)^2 + x^2f(x)^2 + \frac{\nu^2-1/4}{x^2}f(x)^2 \ dx.
    \end{align*}
    This last min-max is exactly the $k$-th eigenvalue $\mu_k$ of $G$. The bound given here is large, but thanks to Lemma \ref{lemma: counting function}, using the same idea as in the first proof, the inequality is actually strict.  
\end{proof}

\subsection{Proof of the second upper bound \eqref{eqn: lower bound eigenvalues 2}}\label{section: proof general second bound}

Before proving the second lower bound \eqref{eqn: lower bound eigenvalues 2} in Theorem \ref{thm: lower bound eigenvalues}, we first need the following estimate. 

\begin{proposition}\label{prop: distance eigenvalue zero bessel}
    For every $\xi > 0$ and $k \geq 0$, we have 
    \begin{align}
        \left| \lambda_{\xi,k} - (j_{\nu,k})^2 \right| \leq \xi^2,
    \end{align}
    where $j_{\nu,k}$ is the $(k+1)$-th zero of the Bessel function of the first kind $J_\nu$.
\end{proposition}

\begin{remark}
    The $j_{\nu,k}^2$ are the eigenvalues of the  Dirichlet Laplacian on the unit disk. Proposition \ref{prop: distance eigenvalue zero bessel} is a comparison principle between these eigenvalues and those of the Laplacian.
\end{remark}

\begin{proof}
    We apply \cite[Theorem 11, p. 104]{dautray2012mathematical}. We see the harmonic potential $\xi^2x^2$ as a zeroth-order perturbation in the sense of quadratic forms of the operator   
    \begin{align}
            G_0 = -\partial_x^2 + \frac{\nu^2 - 1/4}{x^2},  
    \end{align}
    with the same domain as those described at the beginning of Section \ref{section: link Kummer eigenfunction}.\\
    
    \cite[Theorem 11, p. 104]{dautray2012mathematical} states that 
    \begin{align}
        |\lambda_{\xi,k} - \lambda_{0,k}| \leq \xi^2.
    \end{align}
    The eigenvalues of $G_0$ have already been extensively studied in the literature. For instance, from \cite[Proposition III.1]{martinez2018cost}, we know that (be careful that in \cite[Eq. (III.1)]{martinez2018cost}, they use a different convention for the equation),
    \begin{align}
        \lambda_{0,k} = (j_{\nu,k})^2, \quad k \geq 0,
    \end{align}
    where $j_{\nu,k}$ is the $(k+1)$-th zero of the Bessel function of the first kind $J_\nu$ (usually these zeros are indexed from $k = 1$, but in this paper our indexes start from $k=0$). The proposition follows.
\end{proof}

We now prove the second bound \eqref{eqn: lower bound eigenvalues 2} in Theorem \ref{thm: lower bound eigenvalues}.

\begin{proof}[Proof of the second bound \eqref{eqn: lower bound eigenvalues 2} in Theorem \ref{thm: lower bound eigenvalues}]
    
From \eqref{eqn: lower bound eigenvalues 1} in Theorem \ref{thm: lower bound eigenvalues}, we have that for every $k \geq 0$ and $\xi > 0$, 
\begin{align*}
    \lambda_{\xi,k} \geq 4\xi k.
\end{align*}
Therefore, given any $c> 0$, for any $\xi > 0$, and for any $k \leq 4\xi/c$, we have 
\begin{align}\label{eqn: lower bound k^2 in proof of th}
    \lambda_{\xi,k} \geq ck^2.
\end{align}

On the other hand, from Proposition \ref{prop: distance eigenvalue zero bessel}, we have, for every $\xi > 0$ and every $k \geq 0$,
\begin{align}\label{eqn: lower bound eigenvalue from distance to bessel}
    \lambda_{\xi,k} \geq (j_{\nu,k})^2 - \xi^2, 
\end{align}
where we recall that $j_{\nu,k}$ is the $(k+1)$-th zero of the Bessel function of the first kind. These zeros satisfy, as $k \rightarrow + \infty$ (see e.g. \cite[Chapter XV]{watson1922treatise}), 
\begin{align}\label{eqn: asymptotic zero bessel}
    j_{\nu,k} = \left( k + \frac{\nu}{2} + \frac{3}{4} \right)\pi + O\left( \frac{1}{k} \right).
\end{align}

Thus, from \eqref{eqn: lower bound eigenvalue from distance to bessel} and \eqref{eqn: asymptotic zero bessel}, at fixed $\nu \geq 0$, for any $\delta > 0$ small, there exists $\xi_0 > 0$, such that for every $\xi \geq \xi_0$, for every $k \geq 4\xi/c$, we have 
\begin{align}
  \lambda_{\xi,k} \geq k^2 \pi^2 - (1+\delta)\xi^2 > 0.  
\end{align}

We additionally want that 
\begin{align*}
  k^2 \pi^2 - (1+\delta)\xi^2 \geq ck^2.  
\end{align*}

This is equivalent to 
\begin{align}\label{eqn: lower bound k in proof proposition 3.2}
    k \geq \sqrt{\frac{1+\delta}{\pi^2 - c}}\xi.
\end{align}
Thus, we ask that $c > 0$ satisfies $c < \pi^2$ and 
\begin{align}\label{eqn: in proof proposition 3.2}
    \sqrt{\frac{1+\delta}{\pi^2-c}} = \frac{4}{c},
\end{align}
so that \eqref{eqn: lower bound k in proof proposition 3.2} becomes $k \geq 4\xi/c$ and thus \eqref{eqn: lower bound k^2 in proof of th} holds for every $k \geq 0$. Equation \eqref{eqn: in proof proposition 3.2} is equivalent to   
\begin{align}
    \frac{(1+\delta)}{16}c^2 + c - \pi^2 = 0. 
\end{align}
One may show that a solution of this equation that satisfies $0<c<\pi^2$ exists. This concludes the proof by choosing such a value of $c$.
\end{proof}

\section{Upper bound on the low-lying eigenvalues: proof of Theorem \ref{thm: upper bound eigenvalues lower range}}\label{section: upper bound eigenvalues}

In this section, we study the low-lying eigenvalues of $G_\xi$. That is, for $\tau \in (0,1)$, the eigenvalues indexed by $k \leq \left\lfloor \frac{\tau \xi}{4} \right\rfloor$. \\

We show that these eigenvalues of $G_\xi$ stabilize exponentially fast, and uniformly, near the values $\xi \mu_k = 4\xi k + 2\xi(1+\nu)$, where we recall that the $\mu_k$'s are the eigenvalues of $G$ defined by \eqref{eqn: eigenvalues of G} in Proposition \ref{prop: eigv and eigenf of G on R+}. We give an upper bound exponentially close to the first lower bound \eqref{eqn: lower bound eigenvalues 1} in Theorem \ref{thm: lower bound eigenvalues}. \\

The results of this section are very much analogous to the spectral analysis done for the Grushin operator in Fourier, with eigenvalues in the lower regime, in \cite[Section 2.1.2]{allonsius2021analysis}. In particular, Theorem \ref{thm: upper bound eigenvalues lower range} is the analog of \cite[Proposition 2.5]{allonsius2021analysis}. Although the technical lemmas and their proofs are similar to those of \cite{allonsius2021analysis}, the overall approach here will be slightly different. We choose a WKB-type approach to the problem, while in \cite[Proposition 2.5]{allonsius2021analysis}, their proof is more constructive.  

\subsection{Proof of Theorem \ref{thm: upper bound eigenvalues lower range}}\label{section: proof theorem upper bound eigenvalues}

Our approach relies on the following well-known Lemma for self-adjoint operators. 

\begin{lemma}\label{lemma: spectral distance}
    Let $A$ be a self-adjoint operator on a Hilbert space $H$, with domain $D(A)$, and $\lambda \in \mathbb{R}$. Then, 
    \begin{align}
        \operatorname{dist}(\lambda, \sigma(A)) \leq \frac{\|(A-\lambda)u\|}{\|u\|}, \quad \text{for every } u \in D(A),
    \end{align}
where $\sigma(A)$ denotes the spectrum of $A$.
\end{lemma}

The idea is to find functions $u_{\xi,k}$, called quasi-modes (or pseudo-eigenfunctions), that act like approximate eigenfunctions for our operator $G_\xi$, and are associated to approximate eigenvalues, called pseudo-eigenvalues, $\tilde{\mu}_{\xi,k}$. Then, making them elements of $D(G_\xi)$, we apply the above Lemma. The usual WKB analysis allows one to construct such objects, but in the case where the potential is singular, it is much more complicated. However, by dilatation, we see that we can use the eigenfunctions of $G$ on $\mathbb{R}^+$, namely, the Laguerre polynomials. \\

Indeed, recall from Section \ref{section: proof general first bound minmax} that we dilated our eigenproblem and wrote 
\begin{align*}
    \frac{\lambda_{\xi,k}}{\xi} = \tilde{\lambda}_{\xi,k}, 
\end{align*}
where $\tilde{\lambda}_{\xi,k}$ is the $k$-th eigenvalue of $\tilde{G}_\xi$ on $\Omega_\xi = (0,\sqrt{\xi})$ defined by \eqref{eqn: definition tilde G_xi}. When $\xi$ becomes large positive, we can expect $\tilde{\lambda}_{\xi,k}$ to stabilize near the eigenvalues $\mu_k$ of $G$ on $\mathbb{R}^+$ for small values of $k \geq 0$. The eigenfunctions $\Phi_k$ of $G$ are given in Proposition \ref{prop: eigv and eigenf of G on R+}. Since the parameter $a_k$ in the definition of $\Phi_k$ is a negative integer, the Kummer function is a polynomial by definition, and $\Phi_k$ can be expressed in term of the (generalized) Laguerre polynomials. The $n$-th Laguerre polynomial of order $\alpha > - 1$ is defined by
\begin{align}\label{eqn: definition laguerre polynomial}
    L_n^{(\alpha)}(r) = \sum_{j=0}^n (-1)^j \binom{n + \alpha}{n-j}\frac{r^j}{j!}, \quad r \in (0,+\infty), \quad n \in \mathbb{N}.
\end{align}
It is solution of the Strum-Liouville eigenvalue problem
\begin{align}
    r^{-\alpha}e^{r}\left( r^{1+\alpha}e^{-r}(L_n^{(\alpha)}(r))'\right) + nL_n^{(\alpha)}(r) = 0.
\end{align}

We have the link between the Kummer functions \eqref{eqn: def Kummer function} and the Laguerre polynomials $L_k^{(\alpha)}$ (\cite[(13.6.19)]{OlverHandbook2010}),

\begin{align}
    (1+\alpha)_n M(-n, \alpha + 1, r) = n!L_n^{(\alpha)}(r),
\end{align}

from which we deduce
\begin{align}\label{eqn: expression eigenfunctions G laguerre}
    \Phi_k(x) = e^{-x^2/2}x^{1/2 + \nu} \frac{k!}{(1+\nu)_k}L_k^{(\nu)}(x^2).
\end{align}

But we cannot use $\Phi_k$ as is, since it does not satisfy the zero boundary condition at $x = 1$ to belong to $D(\tilde{G}_\xi)$. Thus, we subtract correctly to it its boundary value at $x = 1$ in order to satisfy the condition at $x = 1$ and keep $\Phi_k(0) = 0$. Hence, to prove Proposition \ref{thm: upper bound eigenvalues lower range}, using Lemma \ref{lemma: spectral distance} with the operator $\tilde{G}_\xi$, we will choose as pseudo-eigenfunctions \footnote{In the case of Robin boundary condition $u' = \gamma u$ at $x=\sqrt{\xi}$, one can look for quasi-modes of the form $\phi_{\xi,k}(x) = \Phi_k(x) - Cx^{1/2 + \nu}$, with
\begin{align*}
    C = \frac{\Phi_k'(\sqrt{\xi}) - \gamma\Phi_k(\sqrt{\xi})}{\left(\frac{1}{2}+\nu\right)\xi^{-1/4+\nu/2} - \gamma \xi^{1/4+\nu/2}}.
\end{align*}
Thanks to the recurrence relations satisfied by Laguerre polynomials, $\Phi_k'\left(\sqrt{\xi} \right)$ is expressed as a sum of Laguerre polynomials, multiplied by exponentials and powers of $\xi$, of orders $k$ and $k-1$. The computations presented in the present work can therefore be adapted to this case, using Lemmas \ref{lemma: upper bound eigenf of G low regime} and \ref{lemma: lower bound norm eignf G on R}.} 
\begin{align}\label{eqn: def quasimodes}
    \phi_{\xi,k}(x) = \Phi_k(x) - \frac{x^{1/2 + \nu}}{\xi^{1/4 + \nu/2}}\Phi_k \left( \sqrt{\xi} \right) \in D(\tilde{G}_\xi),
\end{align}
and our pseudo-eigenvalues will be the $\mu_k$'s. \\

We therefore need to estimate from below the $L^2(\Omega_\xi)$-norm of $\phi_{\xi,k}$, and from above the values of $\Phi_k \left( \sqrt{\xi} \right)$. These estimates are given in the following two technical lemmas, that are respectively analogous to \cite[Lemma 2.3 and 2.4]{allonsius2021analysis}.

\begin{lemma}\label{lemma: upper bound eigenf of G low regime}
    Let $\tau \in (0,1)$. There exists $C_1,C_2 > 0$ and $\xi_\tau > 0$, such that for every $\xi \geq \xi_\tau$, for every $k \leq \lfloor \frac{\tau \xi}{4} \rfloor$, we have 
    \begin{align}
        \left| \Phi_k \left( \sqrt{\xi} \right) \right| \leq C_1\frac{\Gamma(1+\nu)}{2}\xi^{-\nu/2 - 1/4} e^{-C_2\xi}.
    \end{align}
\end{lemma}

\begin{lemma}\label{lemma: lower bound norm eignf G on R}
    For every $\delta \in (0,1)$, there exists $\xi_\delta > 0$, such that for every $\xi \geq \xi_\delta$, for every $k \leq \lfloor \frac{\xi}{4} \rfloor$, we have 
    \begin{align}
        \int_0^{\sqrt{\xi}} \Phi_k(x)^2 \ dx \geq (1-\delta) \frac{\Gamma(\nu + 1)}{2}\frac{k!}{(1+\nu)_k}.
    \end{align}
\end{lemma}

Observe that we have not normalized these eigenfunctions yet, as this is not a necessity for the proofs, and it does not simplify the presentation, nor the computations. However, let us still give the norm of $\Phi_k$ as we may see it appear in the proofs. We know from \cite[(5.1.1)]{szeg1939orthogonal} that
\begin{align*}
    \int_0^\infty e^{-r}r^\nu L_k^{(\nu)}(r)^2 \ dr = \Gamma(\nu + 1)\binom{k+\nu}{k} = \frac{\Gamma(\nu + k + 1)}{k!}.
\end{align*}
Making the change of variable $r = x^2$, we readily obtain that 
\begin{align}
     \int_0^\infty \Phi_k(x)^2 \ dx = \frac{1}{2}\left( \frac{k!}{(1+\nu)_k} \right)^2 \frac{\Gamma(\nu + k + 1)}{k!} = \frac{1}{2}\frac{k!}{(1+\nu)_k^2}  \Gamma(\nu + k + 1).
\end{align}
From \cite[(5.2.5)]{OlverHandbook2010}, we have 
\begin{align*}
    \frac{\Gamma(\nu + k + 1)}{(1+\nu)_k} = \Gamma(1+\nu) > 1, 
\end{align*}
which is $\nu!$ if $\nu$ is an integer. It follows that 
\begin{align}\label{eqn: norm eigenfunction on R}
    \|\Phi_k(x)\|_{L^2(\mathbb{R}^+)}^2 = \frac{\Gamma(\nu + 1)}{2}\frac{k!}{(1+\nu)_k} .
\end{align} 

Observe that this is constant with respect to $k$ when $\nu = 0$. \\

We first prove Theorem \ref{thm: upper bound eigenvalues lower range}, and then the two technical Lemmas. 

\begin{proof}[Proof of Theorem \ref{thm: upper bound eigenvalues lower range}]
    Let $\tau \in (0,1)$. We use in this proof the notation $C>0$ for a universal constant that may differ from line to line. Recall that we want to use Lemma \ref{lemma: spectral distance} with 
    \begin{align*}
        \phi_{\xi,k}(x) = \Phi_k(x) - \frac{x^{1/2+\nu}}{\xi^{1/4 + \nu/2}}\Phi_k \left( \sqrt{\xi} \right). 
    \end{align*}
    Observe that 
    \begin{align*}
        \left(-\partial_x^2 + x^2 + \frac{\nu^2 - 1/4}{x^2} \right)(x^{1/2+\nu}) = x^{5/2 + \nu}.
    \end{align*}
    Thus, we have
    \begin{align*}
        \tilde{G}_\xi\phi_{\xi,k}(x) = \mu_k \Phi_k(x) - \frac{\Phi_k \left( \sqrt{\xi} \right)}{\xi^{1/4 + \nu/2}}x^{5/2 + \nu}.
    \end{align*}
    Applying Lemma \ref{lemma: spectral distance}, we obtain 
    \begin{align}\label{eqn: application spectral distance to quasimode}
        \operatorname{dist}(\mu_k, \sigma(\tilde{G}_\xi)) &\leq \frac{\|(\tilde{G}_\xi-\mu_k)\phi_{\xi,k}\|}{\|\phi_{\xi,k}\|} \leq \frac{1}{\left\|\phi_{\xi,k} \right\|}  \left| \frac{\Phi_k \left( \sqrt{\xi} \right)}{\xi^{1/4 + \nu/2}} \right| \left( \left\| x^{5/2 + \nu} \right\| + \mu_k \left\| x^{1/2 + \nu} \right\| \right),
    \end{align}
    where we consider the $L^2(\Omega_\xi)$-norm.\\

    We compute, using that $\mu_k = 4k + 2(1+\nu)\leq \tau \xi + 2(1+\nu) \leq  \xi$ for large $\xi$,
    \begin{align}\label{eqn: upper bound boundary value quasimode2}
        \left\| x^{5/2 + \nu} \right\| + \mu_k \left\| x^{1/2 + \nu} \right\| = \frac{\xi^{(3+\nu)/2}}{\sqrt{6+2\nu}} + (4k+2(1+\nu)) \frac{\xi^{(1+\nu)/2}}{\sqrt{2(1+\nu)}} \leq C\xi^{(3+\nu)/2}, 
    \end{align}
    for some $C> 0$. \\
    
    Using Lemma \ref{lemma: upper bound eigenf of G low regime}, we have, for $\xi$ sufficiently large depending on $\tau \in (0,1)$,
    \begin{align}\label{eqn: upper bound boundary value quasimode1}
        \left| \frac{\Phi_k \left( \sqrt{\xi} \right)}{\xi^{1/4 + \nu/2}}\right| \leq C_1 \frac{\Gamma(1+\nu)}{2} \xi^{-1/2-\nu} e^{-C_2\xi}. 
    \end{align}
    
    On the other hand, using Lemmas \ref{lemma: upper bound eigenf of G low regime} and \ref{lemma: lower bound norm eignf G on R}, for $\xi$ large enough depending on $\delta \in (0,1)$,
    \begin{align*}
        \left\|\phi_{\xi,k} \right\|^2 &= \int_0^{\sqrt{\xi}} \left(\Phi_k(x) - \frac{x^{1/2+\nu}}{\xi^{1/4 + \nu/2}}\Phi_k \left( \sqrt{\xi} \right)\right)^2 \ dx \\
        &\geq \int_0^{\sqrt{\xi}} \frac{1}{2}\Phi_k(x)^2 - \frac{x^{1 +2\nu}}{\xi^{1/2 + \nu}}\Phi_k \left( \sqrt{\xi} \right)^2 \ dx \\
        &\geq  (1-\delta) \frac{\Gamma(\nu + 1)}{4}\frac{k!}{(1+\nu)_k} - C_1^2\frac{\xi^{-\nu}}{2+2\nu} \left( \frac{\Gamma(1+\nu)}{4} \right)^2 e^{-2C_2 \xi}.
    \end{align*}
    If $\nu = 0$, the first term in the right-hand side is constant, and we obtain, for any $C \in (0,1)$ arbitrarly close to $1$, if $\xi$ is taken large enough, 
    \begin{align}\label{eqn: lower bound norm quasimode nu = 0}
        \left\|\phi_{\xi,k} \right\|^2 \geq C\frac{(1-\delta)}{4}.
    \end{align}
    
    If $\nu > 0$, then the first term in the right-hand side is strictly decreasing with respect to $k$, and we obtain
    \begin{align}
        \left\|\phi_{\xi,k} \right\|^2 \geq (1-\delta) \frac{\Gamma(\nu + 1)}{4}\frac{\left\lfloor \frac{\tau \xi}{4} \right\rfloor!}{(1+\nu)_{\left\lfloor \frac{\tau \xi}{4} \right\rfloor}} - C_1^2\frac{\xi^{-\nu}}{2+2\nu} \left( \frac{\Gamma(1+\nu)}{4} \right)^2 e^{-2C_2 \xi}.
    \end{align}
    Using Stirling approximation, we have, as $k \rightarrow + \infty$, 
    \begin{align*}
        \frac{k!}{(1+\nu)_k} &= \frac{k!}{\Gamma(1+\nu +k)}\Gamma(1+\nu) \sim \frac{\sqrt{2k\pi}k^k e^{-k}}{\sqrt{2\pi} k^{1/2 +\nu + k}e^{-k}} = k^{-\nu}\Gamma(1+\nu).
    \end{align*}
    Thus, for $\xi$ sufficiently large, for some $C > 0$, 
    \begin{align}\label{eqn: lower bound norm quasimode nu > 0}
        \left\|\phi_{\xi,k} \right\|^2 \geq C(1-\delta) \xi^{-\nu}
    \end{align}
    Hence, under any consideration for $\nu \geq 0$, for some $C > 0$ that differs with the preceding constants, for every $\xi$ taken large enough, we obtain from \eqref{eqn: application spectral distance to quasimode}, \eqref{eqn: upper bound boundary value quasimode2}, \eqref{eqn: upper bound boundary value quasimode1}, \eqref{eqn: lower bound norm quasimode nu = 0} and \eqref{eqn: lower bound norm quasimode nu > 0},  
    \begin{align}\label{eqn: application spectral distance to quasimode 2}
        \operatorname{dist}(\mu_k, \sigma(\tilde{G}_\xi)) \leq C \xi e^{-C_2\xi}.
    \end{align}
    
    Hence, for every $k \leq \left\lfloor \frac{\tau \xi}{4} \right\rfloor$, we have, for $\xi$ large enough, at least one eigenvalue, that we denote by $E_{\xi,k}$, of $\tilde{G}_\xi$ that is exponentially close to $\mu_k$. To conclude the proof, we need to make sure that this eigenvalue is the only one close to $\mu_k$, and that we have not forgotten any eigenvalue. In other words, that $E_{\xi,k}$ is in fact the $k$-th eigenvalue of $\tilde{G}_\xi$, \textit{i.e.} $E_{\xi,k} = \tilde{\lambda}_{\xi,k}$. This can be seen directly from Theorem \ref{thm: lower bound eigenvalues}. \\

    Indeed, we first look at $k=0$. We see from the lower bound \eqref{eqn: lower bound eigenvalues 1} given in Theorem \ref{thm: lower bound eigenvalues}, that $\tilde{\lambda}_{\xi,0} > \mu_0$. Hence, for $k=0$, an eigenvalue that may satisfy $|\lambda - \mu_0| \leq C \xi e^{-C_2 \xi}$ is $\tilde{\lambda}_{\xi,0}$. Since $\tilde{\lambda}_{\xi,1} > \mu_1 = \mu_0 + 4$, we see that it is the only one for $\xi$ large, at least such that $C\xi e^{-C_2 \xi} < 4$, and so $E_{\xi,0} = \tilde{\lambda}_{\xi,0}$. Iterating for every $k$ under consideration, we obtain that $E_{\xi,k} = \tilde{\lambda}_{\xi,k}$. This concludes the proof.
\end{proof}

\subsection{Proof of Theorem \ref{thm: upper bound eigenvalues lower range fixed k}}\label{section: proof theorem upper bound fixed k}

We follow the same strategy as for Theorem \ref{thm: upper bound eigenvalues lower range}. We use Lemma \ref{lemma: spectral distance} with the quasi-modes $\phi_{\xi,k}$ given by \eqref{eqn: def quasimodes}, with fixed $k$. We therefore obtain 
    \begin{align}\label{eqn: application spectral distance to quasimode2}
        \operatorname{dist}(\mu_k, \sigma(\tilde{G}_\xi)) &\leq \frac{\|(\tilde{G}_\xi-\mu_k)\phi_{\xi,k}\|}{\|\phi_{\xi,k}\|} \leq \frac{1}{\left\|\phi_{\xi,k} \right\|}  \left| \frac{\Phi_k \left( \sqrt{\xi} \right)}{\xi^{1/4 + \nu/2}}\right| \left( \left\| x^{5/2 + \nu} \right\| + \mu_k \left\| x^{1/2 + \nu} \right\| \right),
    \end{align}
where we consider the $L^2\left(0,\sqrt{\xi} \right)$-norm.\\

First we have, using \eqref{eqn: expression eigenfunctions G laguerre},
\begin{align*}
    \left| \frac{\Phi_k \left( \sqrt{\xi} \right)}{\xi^{1/4 + \nu/2}}\right| = e^{-\xi/2}\frac{k!}{(1+\nu)_k}L_k^{(\nu)}\left(\xi \right).
\end{align*}

The term $L_k^{(\nu)}(x)$, defined in \eqref{eqn: definition laguerre polynomial}, is polynomial, with its monomial of highest degree given by $(-1)^k x^k/k!$. Thus, letting $\delta > 0$, there exists $\xi_\delta > 0$ such that for every $\xi \geq \xi_\delta$ we have

\begin{align}\label{eqn: proof fixed k 1}
    \left| \frac{\Phi_k \left( \sqrt{\xi} \right)}{\xi^{1/4 + \nu/2}}\right| \leq (1+\delta) e^{-\xi/2}\frac{\xi^k}{(1+\nu)_k}.
\end{align}

Next we compute, recalling that $\mu_k = 4k + 2(1+\nu)$,
\begin{align}\label{eqn: proof fixed k 2}
    \left\| x^{5/2 + \nu} \right\| + \mu_k \left\| x^{1/2 + \nu} \right\| = \frac{\xi^{(3+\nu)/2}}{\sqrt{6+2\nu}} + (4k+2(1+\nu)) \frac{\xi^{(1+\nu)/2}}{\sqrt{2(1+\nu)}}. 
\end{align}

Finally, we bound from below the term $\| \phi_{\xi,k} \|$ using Lemma \ref{lemma: lower bound norm eignf G on R}, since it cannot be made more precise in the present case of $k$ fixed. We have, for $\xi$ possibly larger than in \eqref{eqn: proof fixed k 1}, 
\begin{align}\label{eqn: proof fixed k 3}
    \| \phi_{\xi,k} \| \geq (1-\delta) \sqrt{\frac{\Gamma(\nu + 1)}{2}\frac{k!}{(1+\nu)_k}}.
\end{align}

Thus, combining \eqref{eqn: application spectral distance to quasimode2}, \eqref{eqn: proof fixed k 1}, \eqref{eqn: proof fixed k 2} and \eqref{eqn: proof fixed k 3}, we obtain, for every $\delta \in (0,1)$, for every $\xi$ sufficiently large,  
\begin{align}
    \operatorname{dist}(\mu_k, \sigma(\tilde{G}_\xi)) &\leq \frac{1+\delta}{1-\delta} \sqrt{\frac{2}{\Gamma(1+\nu)}\frac{(1+\nu)_k}{k!}}\xi^{k + (1+\nu)/2} \left( \frac{\xi}{\sqrt{6+2\nu}} +  \frac{4k+2(1+\nu)}{\sqrt{2(1+\nu)}} \right) e^{-\xi/2}.
\end{align}

This concludes the proof of Theorem \ref{thm: upper bound eigenvalues lower range fixed k} by taking $\epsilon > 0$ such that $1+\epsilon = (1+\delta)/(1-\delta)$.

\subsection{Proof of technical lemmas}\label{section: proof technical lemmas}

This subsection is dedicated to the proofs of Lemmas \ref{lemma: upper bound eigenf of G low regime} and \ref{lemma: lower bound norm eignf G on R}. To prove these Lemmas, we will need the following Plancherel-Rotach type formulas from \cite[Th. 8.22.8]{szeg1939orthogonal}.
\begin{lemma}\label{lemma: plancherel rotach formula} 
    Let $\alpha$ be arbitrary and real, $\epsilon$ and $\omega$ fixed positive numbers. We have 
    \begin{enumerate}[label=(\roman*), ref = \ref{lemma: plancherel rotach formula}(\roman*)]
        \item \label{lemma: plancherel rotach formula 1} for $r = (4n + 2\alpha + 2)\cos(\theta)^2$, $0 < \epsilon_0 \leq \theta \leq \pi/2 - \epsilon n^{-1}$, 
    \begin{align}
        e^{-r/2}L_n^{(\alpha)}(r) = (-1)^n n^{\alpha/2 - 1/4} \frac{\sin\left[\left(n+\frac{\alpha + 1}{2}\right)(\sin(2\theta) - 2\theta) + \frac{3\pi}{4}\right] + (nr)^{-1/2}O(1) }{(\pi \sin(\theta))^{1/2}r^{\alpha/2 + 1/4}};
    \end{align}
    \item \label{lemma: plancherel rotach formula 2} for $r = (4n + 2\alpha + 2)\cosh(\theta)^2$, $0 < \epsilon \leq \theta \leq \omega$, 
    \begin{align}
        e^{-r/2}L_n^{\alpha}(r) = (-1)^n \frac{n^{\alpha/2 - 1/4}}{2} \frac{e^{(n + (\alpha+1)/2)(2\theta - \sinh(2\theta))}}{(\pi \sinh(\theta))^{1/2}r^{\nu/2 + 1/4}}\left( 1 + O\left(\frac{1}{n}\right) \right).
    \end{align}
    \end{enumerate}
    The big $O$ notation holds uniformly with respect to $\theta$.
\end{lemma}

\begin{proof}[Proof of Lemma \ref{lemma: upper bound eigenf of G low regime}]
    Let $\tau \in (0,1)$. Introduce $\delta > 0$ a small, but fixed, parameter, such that $1+\delta < 1/\tau$, and let $\xi_\tau > 0$ be such that for every $\xi \geq \xi_\tau$ we have
    \begin{align}\label{eqn: assumption 1 xi proof of lemma 4.5}
        &\frac{\mu_{\lfloor \frac{\tau \xi }{4} \rfloor}}{2} + \left(\frac{\mu_{\lfloor \frac{\tau \xi }{4} \rfloor}^2}{4} + 1/4 - \nu^2 \right)^{1/2} < (1+\delta)\mu_{\lfloor \frac{\tau \xi }{4} \rfloor} < \xi,
    \end{align}
    and
    \begin{align}\label{eqn: assumption 2 xi proof of lemma 4.5}
        \left(\frac{1}{\tau} - 1 \right) \mu_{\lfloor \frac{\tau \xi }{4} \rfloor} + \tau\frac{\nu^2 - 1/4}{\mu_{\lfloor \frac{\tau \xi }{4} \rfloor}} \geq 0, \quad \text{if } \nu \in [0,1/2),
    \end{align}
    where we recall that $\mu_k = 4k + 2(1+\nu)$ is the $k$-th eigenvalue of $G$ given in Proposition \ref{prop: eigv and eigenf of G on R+}.\\
    
    Following the proof of \cite[Lemma 2.4]{allonsius2021analysis}, we first want to show that for every $\xi \geq \xi_\tau$, $x \geq \sqrt{\xi}$, and $k \leq \lfloor \frac{\tau \xi }{4} \rfloor$, we have 
    \begin{align}\label{eqn: inequality induction eigenf low regime}
        |\Phi_{k+1}(x)| \geq |\Phi_k(x)|,
    \end{align}
    and we prove it by induction. Then, we will bound from above $\left|\Phi_{\lfloor \frac{\tau \xi }{4} \rfloor}\left(\sqrt{\xi}\right) \right|$ to conclude the proof. \\    
    
    We have that \eqref{eqn: inequality induction eigenf low regime} holds for $k = 1$, since 
    \begin{align*}
        \Phi_0(x) &= e^{-x^2/2}x^{1/2 + \nu} ,\\
        \Phi_1(x) &= e^{-x^2/2}x^{1/2 + \nu} \left( 1 - \frac{x^2}{1+\nu} \right).
     \end{align*}
    Assume now that  \eqref{eqn: inequality induction eigenf low regime} holds with $k$ and $k-1$. We first remark that we have the recurrence relation \cite[(13.3.1)]{OlverHandbook2010} on the confluent hypergeometric functions 
    \begin{align}\label{eqn: recurrence relation Kummer functions}
        (b - a)M(a - 1, b, z) + (2a - b + z)M(a, b, z) - aM(a + 1, b, z) = 0.
    \end{align}
    Using the above recurrence relation with $a = - k$, $b = 1 + \nu$, and $z = x^2$, we obtain
    \begin{align*}
        \Phi_{k+1}(x) &= e^{-x^2/2}x^{1/2 + \nu}M(-(k+1), 1 + \nu, x^2) \\
        &= e^{-x^2/2}x^{1/2 + \nu} \frac{(2k + 1 + \nu - x^2)M(-k, 1+\nu,x^2) - kM(-(k-1),1+\nu,x^2)}{1+\nu + k} \\
        &= \frac{2k + 1 + \nu - x^2}{1+\nu + k}\Phi_k(x) - \frac{k}{1+\nu + k}\Phi_{k-1}(x).
    \end{align*}
    Now, we observe that the sign of $\Phi_k(x)$ is given by the sign of $M(-k,1+\nu,x^2)$. That is, by the sign of $L_k^{(\nu)}(x^2)$ from the identity \eqref{eqn: expression eigenfunctions G laguerre}. The sign of $\Phi_k$ is therefore constant from the moment that $x^2$ is greater than the largest root of $L_k^{(\nu)}$ on $(0,+\infty)$. Denote the largest root of the Laguerre polynomial $L_k^{(\nu)}$ by $x_k$, for $0 \leq k \leq \lfloor \frac{\tau \xi }{4} \rfloor$. We have that 
    \begin{align}\label{eqn: ordering roots laguerre}
        x_0 < x_1 < ... < x_{\lfloor \frac{\tau \xi }{4} \rfloor},
    \end{align}
    using for example the interlacing properties of the Laguerre polynomials (\cite[Chap. 3]{szeg1939orthogonal}). Moreover, since the sign of $L_k^{(\nu)}(x)$ for $x > x_k$ is given by the sign of its monomial of largest degree, we obtain by definition of the Laguerre polynomials \eqref{eqn: definition laguerre polynomial} that for every $0 \leq k \leq \lfloor \frac{\tau \xi }{4} \rfloor$, for every $x > x_{\lfloor \frac{\tau \xi }{4} \rfloor}$,
    \begin{align}
        \operatorname{sgn}\left( L_k^{(\nu)} \right) (x) = \left\{
        \begin{array}{cc}
             1 & \text{if $k$ is even,}\\
             -1 & \text{if $k$ is odd.}
        \end{array} \right.
    \end{align}
    From \cite[Theorem 6.31.2]{szeg1939orthogonal}, we have $x_k < 2k + 1 + \nu + ((2k + 1 + \nu)^2 + 1/4 - \nu^2)^{1/2}$. Taking into account the definition of $\mu_k$ given in Proposition \ref{prop: eigv and eigenf of G on R+}, this reads 
    \begin{align}
        x_k < \frac{\mu_k}{2} + \left(\frac{\mu_k^2}{4} + 1/4 - \nu^2 \right)^{1/2}.
    \end{align}
    Hence, from \eqref{eqn: ordering roots laguerre}, and by assumption \eqref{eqn: assumption 1 xi proof of lemma 4.5}, for every $\xi \geq \xi_\tau$, we have
    \begin{align}\label{eqn: upper bound largest root laguerre lower regime}
        x_{\lfloor \frac{\tau \xi }{4} \rfloor} < (1+\delta)\mu_{\lfloor \frac{\tau \xi }{4} \rfloor}.
    \end{align}
    Thus, from the expression \eqref{eqn: expression eigenfunctions G laguerre} of $\Phi_k$ in term of the Laguerre polynomials, and using \eqref{eqn: ordering roots laguerre} and \eqref{eqn: upper bound largest root laguerre lower regime}, we obtain that the largest root of $\Phi_k$, that we now denote by $r_k$, satisfies
    \begin{align}\label{eqn: largest root eignfnct bound above}
        r_k < \sqrt{(1+\delta)\mu_{\lfloor \frac{\tau \xi }{4} \rfloor}}, \quad \text{for every } 0 \leq k \leq \left\lfloor \frac{\tau \xi }{4} \right\rfloor. 
    \end{align}
    Moreover, from assumption \eqref{eqn: assumption 1 xi proof of lemma 4.5} we also have $\sqrt{\xi} > \sqrt{(1+\delta)\mu_{\lfloor \frac{\tau \xi }{4} \rfloor}}$. It follows that for every $x \geq \sqrt{\xi}$, $\xi \geq \xi_\tau$, for every $0 \leq k \leq \lfloor \frac{\tau \xi }{4} \rfloor$, we have 
    \begin{align}
        \operatorname{sgn}\left( \Phi_k \right) (x) = \left\{
        \begin{array}{cl}
             1 & \text{if $k$ is even,}\\
             -1 & \text{if $k$ is odd.}
        \end{array} \right.
    \end{align}
    Assume now that $k + 1$ is even. Hence, $\Phi_{k+1}(x)$ and $\Phi_{k-1}(x)$ are positive. The recurrence relation gives in this case, using the recurrence assumption, 
    \begin{align*}
        \Phi_{k+1}(x) &= \frac{2k + 1 + \nu - x^2}{1+\nu + k}\Phi_k(x) - \frac{k}{1+\nu + k}\Phi_{k-1}(x) \\
        &\geq \frac{2k + 1 + \nu - x^2}{1+\nu + k}\Phi_k(x) - \frac{k}{1+\nu + k}\left|\Phi_{k}(x)\right| \\
        &=  \frac{3k + 1 + \nu - x^2}{1+\nu + k}\Phi_k(x).
    \end{align*}
    Now, both $\Phi_k(x)$ and $3k+1+\nu-x^2$ are negative if $x^2 \geq \xi$. Indeed, by assumption \eqref{eqn: assumption 1 xi proof of lemma 4.5},
    \begin{align*}
        \xi \geq \mu_{\lfloor \frac{\tau \xi }{4} \rfloor} \geq \mu_k \geq 3k + 1 + \nu, \quad \text{for every } 0 \leq k \leq \left\lfloor \frac{\tau \xi }{4} \right\rfloor. 
    \end{align*}
    Also, observe that
    \begin{align}\label{eqn: inequality for recurrence lemma 4.6}
        \frac{3k + 1 + \nu - x^2}{1+\nu + k} \leq -1, 
    \end{align}
    if and only if 
    \begin{align*}
        x^2 \geq \mu_k,
    \end{align*}
    which holds true for every $k \leq \lfloor \frac{\tau \xi }{4} \rfloor$ by assumption \eqref{eqn: assumption 1 xi proof of lemma 4.5} since $x^2 \geq \xi$. Hence, it follows that 
    \begin{align}
        \Phi_{k+1}(x) \geq |\Phi_k(x)|.
    \end{align}
    
    Assume now that $k+1$ is odd. Then, $\Phi_{k+1}(x)$ and $\Phi_{k-1}(x)$ are negative, $\Phi_k(x)$ is positive. Since $2k + 1 + \nu - x^2$ is negative, we get that 
    \begin{align*}
        \frac{2k + 1 + \nu - x^2}{1 + \nu + k}\Phi_k(x) &< 0, \\
        -\frac{k}{1 + \nu + k}\Phi_{k-1}(x) &> 0.
    \end{align*}
    Using again the recurrence assumption and recurrence relation, we get 
    \begin{align*}
        \Phi_{k+1}(x) &= \frac{2k + 1 + \nu - x^2}{1+\nu + k}\Phi_k(x) - \frac{k}{1+\nu + k}\Phi_{k-1}(x) \\
        &\leq \frac{3k + 1 + \nu - x^2}{1+\nu + k}\Phi_k(x).
    \end{align*}
    Using once again \eqref{eqn: inequality for recurrence lemma 4.6},
    \begin{align*}
        \Phi_{k+1}(x) \leq -\Phi_k(x) \leq 0,
    \end{align*}
    which translates by 
    \begin{align}
        |\Phi_{k+1}(x)| = - \Phi_{k+1}(x) \geq \Phi_k(x) = |\Phi_k(x)|.
    \end{align}
    This ends the proof of inequality \eqref{eqn: inequality induction eigenf low regime}.\\
    
    In view of \eqref{eqn: inequality induction eigenf low regime}, it is sufficient to bound from above $\left|\Phi_{\lfloor \frac{\tau \xi }{4} \rfloor}\left(\sqrt{\xi}\right) \right|$. Let us first prove that there exists $C_1,C_2 > 0$ such that 
    \begin{align}\label{eqn: upper bound eignf largest k}
        \left|\Phi_{\lfloor \frac{\tau \xi }{4} \rfloor}\left(\sqrt{\xi}\right) \right| \leq \frac{C_1}{2}{\left\lfloor \frac{\tau \xi }{4} \right\rfloor}^{\nu/2 - 1/4}\frac{\lfloor \frac{\tau \xi }{4} \rfloor!}{(1+\nu)_{\lfloor \frac{\tau \xi }{4} \rfloor}} e^{-C_2\xi }.
    \end{align}
    
    For simplicity of the presentation of the end of the proof, let us set 
    \begin{align}\label{eqn: definition K proof lemma 4.5}
        K := \left\lfloor \frac{\tau \xi }{4} \right\rfloor.
    \end{align}
    Observe that if $x \geq \sqrt{(1+\delta)\mu_K}$, then $x^2 + \frac{\nu^2 - 1/4}{x^2} - \mu_K \geq 0$ by assumption \eqref{eqn: assumption 1 xi proof of lemma 4.5}. Indeed, the largest root of $x \mapsto x^2 + \frac{\nu^2 - 1/4}{x^2} - \mu_K$ is exactly given by the square root of the left-hand term in \eqref{eqn: assumption 1 xi proof of lemma 4.5}. \\
    
    Since
    \begin{align*}
        \Phi_K''(x) = \left(x^2 + \frac{\nu^2 - 1/4}{x^2} - \mu_K\right)\Phi_K(x),
    \end{align*}
    the sign of $\Phi_K''(x)$ coincides with the one of $\Phi_K(x)$ on $[\sqrt{(1+\delta)\mu_K},+\infty)$. Since $\Phi_K(x)$ is of constant sign on $[\sqrt{(1+\delta)\mu_K},+\infty)$ by \eqref{eqn: largest root eignfnct bound above}, and it tends to zero as $x$ tends to infinity, we get that $|\Phi_K|$ is convex strictly decreasing on $[\sqrt{\mu_K/\tau},+\infty)$. Since, by assumption \eqref{eqn: assumption 1 xi proof of lemma 4.5}, we have $\sqrt{\xi} > \sqrt{(1+\delta)\mu_K}$, it follows from the monotony property that 
    \begin{align}
        \left| \Phi_K \left( \sqrt{\xi} \right) \right| \leq \left|\Phi_K \left(\sqrt{(1+\delta)\mu_K} \right) \right|.
    \end{align}
    
    We now want to use Lemma \ref{lemma: plancherel rotach formula 2} to bound from above the right-hand term of the above estimate. By the definition \eqref{eqn: expression eigenfunctions G laguerre} of $\Phi_k$ in term of the Laguerre polynomial, we obtain from Lemma \ref{lemma: plancherel rotach formula 2}, for $\epsilon, \omega$ arbitrary positive numbers, $x^2 = (4k + 2\nu + 2)\cosh(\theta)^2$, and $0 < \epsilon \leq \theta \leq \omega$,
    \begin{align}\label{eqn: eigenf Plancherel-Rotach 2}
        \Phi_K(x) = (-1)^K \frac{K!}{(1+\nu)_K} \frac{K^{\nu/2 - 1/4}}{2}\frac{e^{(K + (\nu+1)/2)(2\theta - \sinh(2\theta))}}{(\pi \sinh(\theta))^{1/2}}\left( 1 + O\left(\frac{1}{K}\right) \right).
    \end{align}
    Observe that since $\mu_k = 4k + 2(1+\nu)$, the above formula holds for $x^2 = \mu_k \cosh(\theta)^2$. Choosing $\theta$ such that $\cosh(\theta) = \sqrt{1+\delta}$, and since $2\theta - \sinh(2\theta) < 0$, \eqref{eqn: upper bound eignf largest k} directly follows from \eqref{eqn: eigenf Plancherel-Rotach 2}. \footnote{If we take $\tau = 1$, then $\delta = 0$, and we must have $\theta = 0$, in which case $2\theta = \sinh(2\theta) = \sinh(\theta) = 0$, and \eqref{eqn: eigenf Plancherel-Rotach 2} is not valid.} \\
    
    Now, we observe that, thanks to Stirling approximations, as $K \rightarrow + \infty$ (or equivalently $\xi \rightarrow + \infty$ by definition \eqref{eqn: definition K proof lemma 4.5} of $K$) 
    \begin{align*}
        K^{\nu/2}\frac{K!}{(1+\nu)_K} &= K^{\nu/2} \frac{K!}{\Gamma(1+\nu + K)}\Gamma(1+\nu) \\
        &\sim K^{\nu/2} \frac{\sqrt{2\pi K}K^K e^{-K}}{\sqrt{2\pi K} K^{\nu + K}e^{-K}}\Gamma(1+\nu) \\
        &= K^{-\nu/2}\Gamma(1+\nu).
    \end{align*}
    
    Thus, up to possibly taking $\xi_\tau$ larger than in the preceding proof by induction, for some $\delta > 0$ and for every $\xi > \xi_\tau$, we have
    \begin{align*}
        \frac{C_1}{2}{\left\lfloor \frac{\tau \xi }{4} \right\rfloor}^{\nu/2 - 1/4}\frac{\lfloor \frac{\tau \xi }{4} \rfloor!}{(1+\nu)_{\lfloor \frac{\tau \xi }{4} \rfloor}} e^{-C_2\xi } &\leq C_1\frac{\Gamma(1+\nu)}{2}(1+\delta){\left\lfloor \frac{\tau \xi }{4} \right\rfloor}^{-\nu/2 - 1/4} e^{-C_2\xi},
    \end{align*}
    which concludes the proof.
\end{proof}

\begin{proof}[Proof of Lemma \ref{lemma: lower bound norm eignf G on R}]
    We can follow almost verbatim the proof of \cite[Lemma 2.3]{allonsius2021analysis}. Hence, we outline the proof but omit some computational details. \\
    
    Let $0 < \delta < 1$, and $\epsilon \in (0,\pi/2)$ small depending on $\delta$, to be determined later. For every $\xi > 2(1+\nu)/(1-\tau)$, for every $k \leq \lfloor \xi /4 \rfloor$, we have $\sqrt{\xi} > \sqrt{\mu_k} \geq \cos(\epsilon)\sqrt{\mu_k}$ by definition of $\mu_k$ (see Proposition \ref{prop: eigv and eigenf of G on R+}). Hence, 
    \begin{align*}
        \int_{\Omega_\xi } \Phi_k(x)^2 \ dx = \int_0^{\sqrt{\xi}} \Phi_k(x)^2 \ dx \geq  \int_0^{\cos(\epsilon)\sqrt{\mu_k}} \Phi_k(x)^2 \ dx.
    \end{align*}
    Make the change of variable $x = \sqrt{\mu_k}\cos(\theta)$, $\theta \in (\epsilon,\pi/2)$. We get
    \begin{align}\label{eqn: lower bound change variable integral proof lemma 4.6}
        \int_{\Omega_\xi } \Phi_k(x)^2 \ dx \geq  \int_\epsilon^{\pi/2 - \epsilon} \Phi_k(\sqrt{\mu_k}\cos(\theta))^2 \sqrt{\mu_k}\sin(\theta) \ d\theta.
    \end{align}
    Let $\theta$ satisfy $0 < \epsilon \leq \theta \leq 1-\epsilon < \pi/2$. We can apply Lemma \ref{lemma: plancherel rotach formula 1} which gives, for $x^2 = (4k + 2(1+\nu))\cos(\theta)^2 = \mu_k \cos(\theta)^2$ by definition of $\mu_k$ from Proposition \ref{prop: eigv and eigenf of G on R+}, and using the expression \eqref{eqn: expression eigenfunctions G laguerre} of $\Phi_k$ in term of the Leguerre polynomials,
    \begin{align}\label{eqn: eigenf Plancherel-Rotach 1}
        \Phi_k(x) = (-1)^k k^{\nu/2 - 1/4}\frac{k!}{(1+\nu)_k} \frac{\sin\left[\left(k+\frac{\nu + 1}{2}\right)(\sin(2\theta) - 2\theta) + \frac{3\pi}{4}\right] +x^{-1} k^{-1/2}O(1) }{(\pi \sin(\theta))^{1/2}}.
    \end{align}
    Thus, incorporating \eqref{eqn: eigenf Plancherel-Rotach 1} into \eqref{eqn: lower bound change variable integral proof lemma 4.6}, we obtain 
    \begin{align*}
        \int_{\Omega_\xi } \Phi_k(x)^2 \ dx \geq k^{\nu - 1/2}\left(\frac{k!}{(1+\nu)_k}\right)^2 \frac{\sqrt{\mu_k}}{\pi} \int_\epsilon^{\pi/2 - \epsilon} &\left[\sin\left[\left(k+\frac{\nu + 1}{2}\right)(\sin(2\theta) - 2\theta) + \frac{3\pi}{4}\right] \right. \\ 
        & \left.  + (\sqrt{\mu_k}\cos(\theta))^{-1} k^{-1/2}O(1)\right]^2 \ d\theta.
    \end{align*}
    This gives, since $\cos(\theta)$ is bounded, 
    \begin{align*}
     \int_{\Omega_\xi } \Phi_k(x)^2 \ dx \geq \frac{2}{\pi}k^{\nu} \left(\frac{k!}{(1+\nu)_k} \right)^2  \int_\epsilon^{\pi/2 - \epsilon} \left[\sin\left[\left(k+\frac{\nu + 1}{2}\right)(\sin(2\theta) - 2\theta) + \frac{3\pi}{4}\right] +O\left( \frac{1}{k} \right) \right]^2 \ d\theta.
    \end{align*}
    First, we observe here that, as in the proof of Lemma \ref{lemma: upper bound eigenf of G low regime}, thanks to Stirling approximations, as $k \rightarrow + \infty$, 
    \begin{align*}
        k^\nu\frac{k!}{(1+\nu)_k} &= k^\nu \frac{k!}{\Gamma(1+\nu +k)}\Gamma(1+\nu) \\
        &\sim k^\nu \frac{\sqrt{2k\pi}k^k e^{-k}}{\sqrt{2\pi} k^{1/2 +\nu + k}e^{-k}}\Gamma(1+\nu) \\
        &= \Gamma(1+\nu).
    \end{align*}
    Hence, given $\delta \in (0,1)$ in the beginning of the proof, there exists $K_1 \in \mathbb{N}$ such that for every $ K \leq k \leq \lfloor \xi /4 \rfloor$, we have 
    \begin{align*}
     \int_{\Omega_\xi } \Phi_k(x)^2 \ dx \geq \frac{k!}{(1+\nu)_k} \frac{2\Gamma(1+\nu)(1-\delta)}{\pi} \int_\epsilon^{\pi/2 - \epsilon} \left[\sin\left[\left(k+\frac{\nu + 1}{2}\right)(\sin(2\theta) - 2\theta) + \frac{3\pi}{4}\right] +O\left( \frac{1}{k} \right) \right]^2 \ d\theta.
    \end{align*}
    Also, analogous computations to the proof of \cite[Lemma 2.3]{allonsius2021analysis} give that given $\delta \in (0,1)$ in the beginning of the proof, there exists $K_2 \in \mathbb{N}$, such that for every $ k \geq K$, we have,
    \begin{align*}
        \frac{2}{\pi} \int_\epsilon^{\pi/2 - \epsilon} \left[\sin\left[\left(k+\frac{\nu + 1}{2}\right)(\sin(2\theta) - 2\theta) + \frac{3\pi}{4}\right] +O\left( \frac{1}{k} \right) \right]^2 \ d\theta \geq \frac{1 - \delta}{2}.
    \end{align*}
    Hence, it follows, combining both of the above asymptotics, that for every $\max{(K_1,K_2)} \leq k \leq \lfloor \xi / 4 \rfloor$, we have 
    \begin{align*}
         \int_{\Omega_\xi } \Phi_k(x)^2 \ dx \geq  \frac{\Gamma(1+\nu)}{2} \frac{k!}{(1+\nu)_k} (1 - \delta).
     \end{align*}
     For the finite numbers of integers $k \in \{0,...,K-1\}$, we have, by \eqref{eqn: norm eigenfunction on R}, that as $\xi \rightarrow + \infty$,
     \begin{align*}
          \int_{\Omega_\xi } \Phi_k(x)^2 \ dx \rightarrow \| \Phi_k \|_{L^2(\mathbb{R}^+)}^2 = \frac{\Gamma(\nu + 1)}{2}\frac{k!}{(1+\nu)_k},
     \end{align*}
     and the Lemma follows. 
\end{proof}

\section*{Acknowledgments}

The author would like to thank his PhD advisors P. Lissy and D. Prandi for their corrections and suggestions on this paper. The author extends his thanks to M. Morancey for discussions during a conference in Benasque concerning the spectral analysis of his paper \cite{allonsius2021analysis}. The author would also like to thank M. Baur for pointing out the reference \cite{baur2025eigenvalues}, and an anonymous reviewer for valuable comments that helped greatly improve the presentation of this paper. 

\printbibliography

\end{document}